\documentclass[11pt,reqno]{amsart}

\usepackage{amsfonts}
\usepackage{amsmath,amscd}
\usepackage{amssymb}
\usepackage{tikz-cd}
\usepackage{amsthm}
\usepackage{cases}

\usepackage{galois}
\usepackage{chemarrow}
\usepackage[all]{xy}
\usepackage{graphicx}
\usepackage[colorlinks,
            linkcolor=blue,
            linktoc=all,
            anchorcolor=black,
            citecolor=red
            ]{hyperref}	
\usepackage{mathrsfs}
\usepackage{extarrows}
\usepackage{texnames}
\usepackage[nameinlink,capitalize]{cleveref}

\def\End{{\rm End}}

\def\Ker{{\rm Ker}}

\def\Tr{{\rm Tr}}

\theoremstyle{plain}

\newtheorem{theorem}{Theorem}[section]

\newtheorem{proposition/example}[theorem]{Proposition/Example}
\newtheorem{proposition}[theorem]{Proposition}
\newtheorem{corollary}[theorem]{Corollary}
\newtheorem{lemma}[theorem]{Lemma}

\newtheorem{conjecture}[theorem]{Conjecture}

\theoremstyle{definition}
\newtheorem{definition}[theorem]{Definition}
\newtheorem{remark}[theorem]{Remark}

\newtheorem{conjecture/question}[theorem]{Conjecture/Question}

\newtheorem{remark/definition}[theorem]{Remark/Definition}
\newtheorem{definition/notation}[theorem]{Definition/Notation}

\numberwithin{equation}{section}

\begin{document}
\title{\textbf{Curvature Formulas Related to a Family of Stable Higgs Bundles}}

\author{Zhi Hu}

\address{ \textsc{School of Science, Nanjing University of Science and Technology, Nanjing 210094,  China}\endgraf \textsc{Research Institute for Mathematical Sciences, Kyoto University, Kyoto, 606-8502, Japan}\endgraf \textsc{Department of Mathematics, Mainz University, 55128 Mainz, Germany}}

\email{halfask@mail.ustc.edu.cn; huz@uni-mainz.de}

\author{Pengfei Huang}

\address{ \textsc{Mathematisches Institut, Ruprecht-Karls Universit\"at Heidelberg, Im Neuenheimer Feld 205, 69120 Heidelberg, Germany}}

\email{pfhwang@mathi.uni-heidelberg.de; pfhwangmath@gmail.com}


\date{}

\begin{abstract}
In this paper,  we investigate  the geometry of the base complex manifold of an effectively parametrized holomorphic family of   stable Higgs bundles over a fixed  compact K\"{a}hler manifold. The starting point of our study  is Schumacher--Toma/Biswas--Schumacher's curvature formulas for Weil--Petersson-type metrics, in Sect. \ref{sec2}, we give some applications of their formulas on the geometric properties of the base manifold.  In Sect. \ref{sec3}, we calculate the curvature  on the higher direct image bundle, which recovers Biswas--Schumacher's curvature formula. In Sect. \ref{sec4}, we construct a smooth and strongly pseudo-convex complex Finsler metric for the base manifold, the corresponding holomorphic sectional curvature is calculated explicitly.
\end{abstract}

\maketitle
\tableofcontents

\section{Introduction}\label{sec1}
Studying the  moduli space of certain geometric objects from the viewpoint of differential geometry is an important approach to understand the geometry of the moduli space. A basic starting point  is to endow the moduli space with a suitable Riemannian metric. If the parametrized geometric object is equipped with  a metric, in general,  the moduli space could inherit a natural metric as a functional of the metric on the  geometric object. Such metric on the  moduli space is usually called the \emph{Weil--Petersson-type metric}.
There are two typical examples of applying such ideas.

\begin{enumerate}
  \item[(i)] The moduli space of certain compact  polarized complex manifolds:
  \begin{itemize}
    \item For the moduli space  $\mathcal{M}_g$ (and the Teichm\"uller space $\mathcal{T}_g$) of compact Riemann surfaces of genus $g\geq 2$, Ahlfors
 showed that the Weil--Petersson metric on $\mathcal{T}_g$  is a K\"{a}hler metric whose Ricci and holomorphic sectional curvatures are negative \cite{Ah1,Ah2} (see also \cite{FT} by Fischer and Tromba). Later Royden  proved that the holomorphic sectional curvature of the Weil--Petersson metric is bounded away from zero \cite{Ro}. Afterwards Wolpert  showed that the holomorphic sectional curvature of the Weil--Petersson metric is bounded above by
$-\frac{1}{2\pi(g-1)}$ \cite{Wol}, which confirms a conjecture of Royden \cite{Ro}. This immediately implies the moduli space $\mathcal{M}_g$ is  Kobayashi hyperbolic.
    \item Due to the famous results of  Aubin \cite{Au} and Yau \cite{Yau}, every  compact canonically  polarized complex manifold admits a unique K\"{a}hler--Einstein metric of negative Ricci curvature unique up to a positive multiplicative constant. Then  the  moduli space of compact canonically  polarized complex manifolds can be equipped with a Weil--Petersson-type metric, and Siu computed the corresponding curvature tensor \cite{Siu}.
    \item Also by Yau's solution to the Calabi conjecture \cite{Yau}, there is a Weil--Petersson-type metric on the moduli space of compact  polarized Calabi--Yau manifolds. Strominger  gave the curvature formula for the case of Calabi--Yau threefolds  using the Yukawa couplings \cite{Str}.
  \end{itemize}
  \item[(ii)] The moduli space  of stable  bundles over a fixed  compact K\"{a}hler manifold:
  \begin{itemize}
    \item Thanks to many author's (Narasimhan--Seshadri \cite{NS}, Donaldson \cite{Don}, Uhlenbeck--Yau \cite{UY}) work on   the existence of Hermitian--Einstein metrics on  stable vector bundles, we also have the Weil--Petersson-type metric on the  moduli space of stable vector bundles over a fixed  compact K\"{a}hler manifold. In this case, Schumacher and Toma calculated the corresponding curvature tensor \cite{ST}.
    \item The notion of  Higgs field on a vector bundle was introduced by Hitchin \cite{Hit} for Riemann surface case and by Simpson \cite{Sim} in general, they also showed that stable Higgs bundles admit  Hermitian--Einstein metrics. The work of Schumacher and Toma is then generalized to the moduli space of stable Higgs bundles over a fixed  compact K\"{a}hler manifold by Biswas and Schumacher \cite{BS}.
  \end{itemize}
\end{enumerate}

In order to capture more geometric information of moduli spaces, the method mentioned above is developed by constructing certain suitable Finsler metric on the moduli space. If the moduli space has an initial Riemannian  metric, there is a way to construct the Finsler metric by recursive introducing   the higher order parts from  the terms in the curvature with fixed sign. For the moduli space of certain manifolds, this idea has been carried out. For example,
 Schumacher \cite{Sch} and To--Yeung \cite{TY} constructed  a Finsler metric on the moduli space of compact canonically polarized complex manifolds
   based on Siu's curvature formula, and Deng \cite{Den} constructed a Finsler metric on the moduli space of compact  polarized complex manifolds with semi-ample canonical bundles based on Griffiths's curvature formula of Hodge bundles \cite{Gri}.
   As remarkable  applications, they can show the hyperbolicity in certain sense for these moluli spaces by calculating  the holomorphic sectional curvature with respect to the  Finsler metric.

This paper investigates the geometry of  the base complex manifold of an effectively parametrized holomorphic family of   stable Higgs bundles over a fixed  compact K\"{a}hler manifold by calculating the curvature with respect to  suitable Weil--Petersson-type metric or Finsler metric.

\subsection{Setup and Some Notations}

\hspace*{5pt}

Let $X$ be  an $n$-dimensional compact K\"{a}hler manifold equipped with a  K\"{a}hler form $\omega_X=\sqrt{-1}g_{\alpha\bar \beta}dz^\alpha\wedge d\bar z^\beta$ in terms of the local holomorphic coordinates $(z^1,\cdots, z^n)$,  and  let $E$  be a vector bundle over $X$ with a Hermitian metric $h$.
We introduce the following notations
\begin{enumerate}
\item[$\bullet$] For any smooth $\mathrm{End}(E)$-valued $(p,q)$-form $\theta\in \mathcal{A}^{p,q}_X(\mathrm{End}(E))$, its Hermitian conjugate $\theta^{*_h}\in \mathcal{A}_X^{q,p}(\mathrm{End}(E))$ is defined by
$$
h(\theta(u),v)=h(u,\theta^{*h}(v))
$$
for any $u,v\in C^\infty(E)$.
\item[$\bullet$] For any two smooth $\mathrm{End}(E)$-valued $r$-forms $(r\leq n)$ $\varphi,\psi\in\mathcal{A}_X^r(\mathrm{End}(E))$, writing
\begin{align*}
 \varphi&=\sum_{p+q=r}\varphi_{\alpha_1\cdots\alpha_{p}\bar\beta_1\cdots\bar\beta_{q}}dz^{\alpha_1}\wedge \cdots\wedge dz^{\alpha_p}\wedge d\bar z^{\beta_1}\wedge \cdots\wedge d\bar z^{\beta_q},\\
 \psi&=\sum_{p+q=r}\psi_{\gamma_1\cdots\gamma_{p}\bar\delta_1\cdots\bar\delta_{q}}dz^{\gamma_1}\wedge \cdots\wedge dz^{\gamma_p}\wedge d\bar z^{\delta_1}\wedge \cdots\wedge d\bar z^{\delta_q},
\end{align*}
 we define their inner product with respect to the induced metric $\tilde{h}$ on $\mathcal{A}_X^r(\mathrm{End}(E))$ as
\begin{align*}
\tilde{h}(\varphi,\psi)=\sum_{p+q=r}g^{\bar \gamma_1\alpha_1}\cdots g^{\bar\gamma_p\alpha_p}g^{\bar\beta_1\delta_1}\cdots g^{\bar\beta_q\delta_q}\Tr(\varphi_{\alpha_1\cdots\alpha_{p}\bar\beta_1\cdots\bar\beta_{q}}\psi^{*_h}_{\bar\gamma_1\cdots\bar \gamma_{p}\delta_1\cdots\delta_{q}}).
\end{align*}
\item[$\bullet$] For any operator $\Xi: \mathcal{A}_X^{p,q}(\mathrm{End}(E))\to\mathcal{A}^{r,s}_X(\mathrm{End}(E))$, its formal adjoint $\Xi^{\dagger_{\tilde h}}: \mathcal{A}^{r,s}_X(\mathrm{End}(E))\to\mathcal{A}^{p,q}_X(\mathrm{End}(E))$ with respect to the induced metric $\tilde h$ is defined by
$$
\int_X\tilde{h}(\Xi\varphi,\psi)\frac{\omega_X^n}{n!}=\int_X\tilde{h}(\varphi,\Xi^{\dagger_{\tilde h}}\psi)\frac{\omega_X^n}{n!}
$$
for any $\varphi\in\mathcal{A}_X^{p,q}(\mathrm{End}(E))$ and $\psi\in\mathcal{A}_X^{r,s}(\mathrm{End}(E))$.
\end{enumerate}

Let $S$ be an $m$-dimensional  complex manifold, and let $\{(\mathcal{E}_s,\Phi_s)\}_{s\in S}$ be a holomorphic family of stable Higgs bundles of rank $r$ on $X$ parametrized by  $S$, namely a Higgs bundle $(\mathcal{E},\Phi)$ over $\tilde{X}:= X\times S$ such that $(\mathcal{E},\Phi)|_{X\times\{s\}}=(\mathcal{E}_s,\Phi_s)$ for each $s\in S$.   The relative Higgs complex associated to $(\mathcal{E},\Phi)$ is given by
$$
\mathbf{H}_{\mathrm{r}}: 0\longrightarrow\mathcal{E}\xlongrightarrow{\Phi\wedge}\mathcal{E}\otimes\Omega_{\tilde{X}/S}^1\xlongrightarrow{\Phi\wedge}\cdots.
$$
It is known that the $d$-th direct image sheaf $\mathbb{R}^d\pi_*\mathbf{H}_{\mathrm{r}}$ is coherent over $S$, and is thus locally free outside a proper analytic subvariety $Z^{(d)}$ of $S$, where $\pi: \tilde{X}\to S$ is the projection onto the second factor. Moreover, $\mathbb{R}^d\pi_*\mathbf{H}_{\mathrm{r}}|_s=\mathbb{H}^d(X,\mathbf{H}_s)$ for any $s\in X\backslash Z^{(d)}$, where 
$$
\mathbf{H}_s: 0\longrightarrow\mathcal{E}_s\xlongrightarrow{\Phi_s\wedge}\mathcal{E}_s\otimes\Omega_X^1\xlongrightarrow{\Phi_s\wedge}\cdots
$$
is the Higgs complex associated to the Higgs bundle $(\mathcal{E}_s,\Phi_s)$.
In particular, if both $X$ and $S$ are Riemann surfaces, then Donagi--Pantev--Simpson's result {\cite[Theorem 3.6]{DPS}} implies that  $\mathbb{R}^d\pi_*\mathbf{H}_\mathrm{r}$ is locally free over $S$.

By the short exact sequence 
$$
0\longrightarrow\pi^*\Omega_S^1\otimes\Omega_{\tilde{X}/S}^{p-1}\longrightarrow\Omega_{\tilde{X}}^p/(\pi^*\Omega_S^2\otimes\Omega_{\tilde{X}/S}^{p-2})\longrightarrow\Omega_{\tilde{X}/S}^p\longrightarrow0,
$$
 these is a  connecting morphism in the corresponding long exact sequence
$$
\rho: \mathbb{R}^1\pi_*\mathbf{H}_{\mathrm{r}}\longrightarrow\mathbb{R}^2\pi_*(\pi^*\Omega_S^1\otimes\mathbf{H}_{\mathrm{r}}[-1])=\Omega_S^1\otimes\mathbb{R}^1\pi_*\mathbf{H}_{\mathrm{r}},
$$
or equivalently,
$$
\rho: TS\longrightarrow\mathrm{End}(\mathbb{R}^1\pi_*\mathbf{H}_{\mathrm{r}}),
$$
hence when restricting at each point $s\in S\backslash Z^{(1)}$, we have a map called the \emph{Kodaira--Spencer map}
$$
\rho_s: T_sS\longrightarrow\mathbb{H}^1(X,\mathbf{EH}_s),
$$
where $\mathbb{H}^1(X,\mathbf{EH}_s)$ is the first hypercohomology of the complex
$$
\mathbf{EH}_s: 0\longrightarrow\mathrm{End}(\mathcal{E}_s)\xlongrightarrow{\Phi_s\otimes\mathrm{Id}+\mathrm{Id}\otimes\Phi_s}\mathrm{End}(\mathcal{E}_s)\otimes\Omega_X^1\longrightarrow\cdots.
$$
Actually, the Kodaira--Spencer map $\rho_s$ can be defined for all $s\in S$ \cite{ST,BS}.

In this paper, we always assume our family $(\mathcal{E},\Phi)$ is \emph{effectively parametrized}, namely the Kodaira--Spencer map $\rho_s$
 is injective for any $s\in S$.

 By the Kobayashi--Hitchin correspondence for Higgs bundles  \cite{Hit,Sim},  on each fiber $(\mathcal{E}_s,\Phi_s)$ for $s\in S$, there is a Hermitian--Einstein metric $h_s$, and the family $\{h_s\}_{s\in S}$ of metrics  induces  a Hermitian metric $h$ on $\mathcal{E}$. We introduce the operators $\mathbf{d}, \mathbf{d}^{\dagger_{\tilde{h}}}, \Box$, which are fiberwisely defined as follows
 \begin{enumerate}
 \item[$\bullet$] $\mathbf{d}|_s=\bar\partial_{\mathcal{E}_s}+\Phi_s$, where  $\bar\partial_{\mathcal{E}_s}$ stands for the holomorphic structure on $\mathcal{E}_s$;
 \medskip
 \item[$\bullet$] $\mathbf{d}^{\dagger_{\tilde h}}|_s=(\bar\partial_{\mathcal{E}_s})^{\dagger_{\tilde {h}_s}}+(\Phi_s)^{\dagger_{\tilde {h}_s}}$;
 \medskip
 \item[$\bullet$] $
 \Box|_s=\mathbf{d}|_s\mathbf{d}^{\dagger_{\tilde h}}|_s+\mathbf{d}^{\dagger_{\tilde h}}|_s\mathbf{d}|_s.
 $
 \end{enumerate}

Let $\mathbf{TH}_s$ be the total complex associated to the Dolbeault resolution of $\mathbf{H}_s$, then the hypercohomology $\mathbb{H}^\bullet(X,\mathbf{H}_s)$ is computed by the usual cohomology $H^\bullet(X,\mathbf{TH}_s)$, and by applying harmonic theory via the metric $h_s$ one finds the unique harmonic representative for the cohomology class in $\mathbb{H}^\bullet(X,\mathbf{H}_s)$ (more detailed can be found in the Section 3 of \cite{BS}).  Let $\mathfrak{t}$ be a harmonic representative  of $\mathbb{H}^d(X,\mathbf{H}_\mathrm{s})$, then we have
$$
\mathbf{d}|_s\mathfrak{t}=\mathbf{d}^{\dagger_{\tilde{h}}}|_s\mathfrak{t}=0,
$$
or equivalently, expressing
$$
 \mathfrak{t}=\sum_{p+q=d}\mathfrak{t}^{p,q}=\sum_{p+q=d}\mathfrak{t}_{\alpha_1\cdots\alpha_{p}\bar\beta_1\cdots\bar\beta_{q}}dz^{\alpha_1}\wedge \cdots\wedge dz^{\alpha_p}\wedge d\bar z^{\beta_1}\wedge \cdots\wedge d\bar z^{\beta_q}
 $$
we have
\begin{align*}
  \bar\partial_{\mathcal{E}_s}\mathfrak{t}^{p,q}+\Phi_s(\mathfrak{t}^{p-1,q+1})=&0,\\
  [\Lambda_{\omega_X},\partial^{h_s}_{\mathcal{E}_s}]\mathfrak{t}^{p-1,q+1}+[\Lambda_{\omega_X},\Phi^{*_h}_s]\mathfrak{t}^{p,q}=&0
\end{align*}
for any $1\leq p, q\leq n$, where the second equation is due to the K\"ahler identity.
 
 Let $U\subseteq S$ be an open neighborhood of $s\in S$ with  local holomorphic coordinates given by $(s^1,\cdots,s^m)$, the  curvature form with respect to the metric $h$ can be expressed locally as\footnote{Throughout this paper, we use the lowercase Greek letters $\alpha,\beta,\gamma,\cdots$ and Roman letters $i,j,k,\cdots$ for the coordinates on $X$ and $S$, respectively.}  
$$
R=R_{\alpha\bar \beta}dz^\alpha\wedge d\bar z^\beta+R_{i\bar \alpha}ds^i\wedge d\bar z^\alpha+R_{\alpha\bar i} d z^\alpha\wedge d\bar s^i+R_{i\bar j}ds^i\wedge d\bar s^j.
$$
 We  write $\Phi=\Phi_\alpha dz^\alpha, \Phi^{*_h}=\Phi^{*_h}_{\bar \alpha} d\bar z^{ \alpha}$. In particular, the Hermitian--Einstein condition at each $s\in S$ is given by
 \begin{align}\label{he}
  g^{\bar \beta\alpha}(R_{\alpha\bar \beta}+[\Phi_\alpha,\Phi^{*_h}_{\bar \beta}])|_s=\lambda_s\cdot \mathrm{Id}_{\mathcal{E}_s}
 \end{align}
 for some real constant $\lambda_s$ determined by the slope of $\mathcal{E}_s$ and the volume of $X$. 
 
 Again by the Dolbeault resolution and harmonic theory, we have the unique harmonic representative for the cohomology class in $\mathbb{H}^\bullet(X,\mathbf{EH}_s)$. In \cite{BS}, Biswas and Schumacher gave the harmonic representative for the class of the image of the Kodaira--Spencer map. More precisely, let $\nabla_i$ and $\nabla_{\bar i}$ be the covariant derivatives along the directions $\frac{\partial}{\partial s^i}$ and $\frac{\partial}{\partial \bar s^i}$ determined by the connection of the metric $h$, respectively, and define a $C^\infty(\End(\mathcal{E}))$-valued 1-form
 $$
 \eta_i=R_{i\bar \alpha}d\bar z^\alpha+\nabla_i\Phi_\alpha dz^\alpha,
 $$
then for any $s\in U$, $\eta_i|_s$ is the harmonic representative of $\rho_s(\frac{\partial}{\partial s^i}|_s)$ in $\mathbb{H}^1(X, \mathbf{EH}_s)$ due to the Hermitian--Einstein condition. In particular, we have
$$
\mathbf{d}|_s(\eta_i|_s)=\mathbf{d}^{\dagger_{\tilde h}}|_s(\eta_i|_s)=0
$$
for any $s\in U$.

\begin{definition}
Let $\xi=\xi^i\frac{\partial}{\partial s^i}$ be a local holomorphic vector field over  $U$.
\begin{enumerate}
  \item The \emph{Kodaira--Spencer image }of $\xi$ is defined as
  $$
  \mathbb{H}(\xi)=\xi^i\eta_i=\xi^i(R_{i\bar \alpha}d\bar z^\alpha+\nabla_i\Phi_\alpha dz^\alpha),
  $$
  namely $\mathbb{H}(\xi)|_s$ is the harmonic representative of $\rho_s(\xi|_s)$ in $\mathbb{H}^1(X, \mathbf{EH}_s)$ for each $s\in U$.
  \item  $\xi$ is called \emph{relatively harmonic} if 
  $$
 \mathbf{d}|_s [\mathbb{H}(\xi)|_s,G|_s(\mathbb{H}(\xi)|_s)^{\dagger_{\tilde{ h}_s}}\mathbb{H}(\xi)_s]=\mathbf{d}^{\dagger_{\tilde h}}|_s[\mathbb{H}(\xi)|_s,G|_s(\mathbb{H}(\xi)|_s)^{\dagger_{\tilde h_s}}\mathbb{H}(\xi)|_s]=0
 $$
 for any $s\in U$, where $G=\Box^{-1}$ denotes the Green operator that corresponds to $\Box$.
 
\end{enumerate}
\end{definition}

\subsection{Main Results}

\hspace*{5pt}

With the notations introduced above, we consider a holomorphic family of stable Higgs bundles of rank $r$ over $X$ effectively parametrized by $S$. The  Weil--Petersson-type metric $G^{\mathrm{WP}}$ on $S$ is defined as follows \cite{BS}
 \begin{align*}
  G_{i\bar j}^{\mathrm{WP}}=G^{\mathrm{WP}}\Big(\frac{\partial}{\partial s^i}, \frac{\partial}{\partial \bar s^j}\Big)=&\int_X\tilde h(\eta_i,\eta_j)\frac{\omega_X^n}{n!}\\
  =&\int_Xg^{\bar \alpha\beta}\Tr(R_{i\bar \alpha}R_{\beta \bar j}+\nabla_i\Phi_\beta\nabla_{\bar j}\Phi^{*_h}_{\bar \alpha})\frac{\omega_X^n}{n!}.
 \end{align*}
The assumption of effective parametrization guarantees the positive-definiteness of $G^{\mathrm{WP}}$. It is easy to check that this metric is K\"{a}hler. Biswas and Schumacher calculated the  curvature of $G^{\mathrm{WP}}$ as \cite{BS}
\begin{align}\label{1}
 \mathcal{R}_{i\bar jk\bar l}=&\int_X\Tr(R_{i
\bar j}\Box R_{k\bar l}+R_{i
\bar l}\Box R_{k\bar j})\frac{\omega_X^n}{n!}\nonumber\\
&-\int_X\Tr([\eta_i\wedge \eta_k]\wedge G([\eta^{*_h}_{\bar j}\wedge\eta^{*_h}_{\bar l}]))\wedge \frac{\omega_X^{n-2}}{(n-2)!},
\end{align}
where $[\bullet\wedge \bullet]$ stands for the exterior product of forms with values in an endomorphism bundle combined with the Lie bracket.

In Sect. \ref{sec3}, we consider the locally free higher direct image sheaf $\mathbb{R}^d\pi_*\mathbf{H}_\mathrm{r}$ of rank $r_d$ over $S\backslash Z^{(d)}$ with respect to the relative Higgs complex $\mathbf{H}_\mathrm{r}$ introduced above. $\mathbb{R}^d\pi_*\mathbf{H}_\mathrm{r}$ is equipped with an $L^2$-metric $H$ as follows
$$
H(\mathfrak{t},\mathfrak{t}'):=\int_X\tilde{h}_s(\mathfrak{t},\mathfrak{t}')\frac{\omega_X^n}{n!},
$$
where  $\mathfrak{t},\mathfrak{t}^\prime$ are  harmonic representatives  of $\mathbb{H}^d(X,\mathbf{H}_\mathrm{s})$, and
$$
\tilde{h}_s(\mathfrak{t},\mathfrak{t}')=\sum_{p+q=d}g^{\bar \gamma_1\alpha_1}\cdots g^{\bar\gamma_p\alpha_p}g^{\bar\beta_1\delta_1}\cdots g^{\bar\beta_q\delta_q}h_s(\mathfrak{t}_{\alpha_1\cdots\alpha_{p}\bar\beta_1\cdots\bar\beta_{q}},\mathfrak{t}'_{\gamma_1\cdots \gamma_{p}\bar\delta_1\cdots\bar\delta_{q}}).
$$

We calculate the corresponding curvature tensor,  which generalizes the results of To--Weng in \cite{TW} and Geiger--Schumacher in \cite{m2}.

\begin{theorem}[= Theorem \ref{curfor}]\label{thm1}
The curvature tensor $\mathfrak{R}$ of $(\mathbb{R}^d\pi_*\mathbf{H}_\mathrm{r},H)$ over $S\backslash Z^{(d)}$ is given by
\begin{align}\label{curf}
\begin{aligned}
 \mathfrak{R}_{a\bar b i\bar j}=\ &\bigg(\frac{1}{r\mathrm{Vol}(X,\omega_X)}\int_X\Tr(R_{i\bar j})\frac{\omega_X^n}{n!}\bigg)\cdot H_{a\bar b}+\int_X \tilde h\big(G\big(\eta_j^{\dagger_{\tilde h}}(t_a)\big),\eta_i^{\dagger_{\tilde h}}(t_b)\big)\frac{\omega_X^n}{n!}\\
 &-\int_X \tilde{h}\big(G\big((\eta_j^{\dagger_{\tilde h}}\eta_i)(t_a)\big),t_b\big)\frac{\omega_X^n}{n!}-\int_X \tilde h\big(G(\eta_i\wedge t_a),\eta_j\wedge t_b\big)\frac{\omega_X^n}{n!}
 \end{aligned}
\end{align}
over an open  neighborhood $U\subseteq S\backslash Z^{(d)}$
for
$a,b=1,\cdots, r_d, i,j=1,\cdots,m$, where
 $\{t_1|_s,\cdots, t_{r_d}|_s\}$ forms a basis of $\mathbb{H}^d(X,\mathbf{H}_s)$ with each $t_a|_s$ being $\Box|_s$-harmonic for any $s\in U$.
\end{theorem}

It is obvious that $G^{\mathrm{WP}}$ can be viewed as a metric on the sheaf $\mathbb{R}^1\pi_*\mathbf{EH}_{\mathrm{r}}$ over $S$, where $\mathbf{EH}_\mathrm{r}$ is the relative Higgs complex associated to the Higgs bundle $(\mathrm{End}(\mathcal{E}),\Phi\otimes\mathrm{Id}+\mathrm{Id}\otimes\Phi)$. Then by applying above curvature formula for $(\mathbb{R}^1\pi_*\mathbf{EH}_{\mathrm{r}}, G^{\mathrm{WP}})$, in other words, replacing $t_a$ by $\eta_i$ in \eqref{curf}, we can recover the curvature formula \eqref{1} of Biswas and Schumacher. 

In Sect. \ref{sec4}, we continue to consider a holomorphic family of stable Higgs bundles of rank $r$ over $X$ effectively parametrized by $S$ from the viewpoint of Finsler geometry. When considering the holomorphic sectional curvature of the Weil--Petersson-type metric $G^{\mathrm{WP}}$, one immediately  finds that the first term on the right hand side of \eqref{1} is semipositive, and the second term becomes seminegative, which provides the higher-order part of certain Finsler metric according to the method of recursion mentioned above. The Finsler metric is constructed by adding this part into the Weil--Petersson-type metric, namely our Finsler metric $F_\kappa$ with a parameter $\kappa\geq0$ is given by 
 $$
F_\kappa=\sqrt[4]{F_{(1)}^4+\kappa F_{(2)}^4},
$$
where
\begin{align*}
 F_{(1)}(v)&=\sqrt{\int_X\tilde h( \mathbb{H}(\xi),\mathbb{H}(\xi))\frac{\omega^n}{n!}}, \\
 F_{(2)}(v)&=\sqrt[4]{\int_X\tilde h( [\mathbb{H}(\xi)\wedge\mathbb{ H}(\xi)],G[\mathbb{H}(\xi)\wedge \mathbb{H}(\xi)])\frac{\omega^n}{n!}}
\end{align*}
for $v=(s,\xi)\in TS$.

We show that $F_\kappa$ is a smooth and strongly pseudo-convex complex  Finsler metric on $S$ (Proposition \ref{spc}), and the corresponding holomorphic sectional curvature is calculated explicitly.

\begin{theorem}[= Theorem \ref{finslercur}]\label{thm2}
Let $\{s^1,\cdots,s^m\}$ be the local holomorphic coordinate on $S$.
The holomorphic sectional curvature of the Finsler metric $F_\kappa$
 is given by

\begin{align*}
 K_{F_\kappa}\Big(\frac{\partial}{\partial s^i}\Big)
 = &\ \bigg(\Big(\int_X\tilde h(\eta_i,\eta_i)\frac{\omega^n}{n!}\Big)^2+\kappa\int_X\tilde h(\mathbf{d}^{\dagger_{\tilde h}}G[\eta_i\wedge \eta_i], \mathbf{d}^{\dagger_{\tilde h}}G[\eta_i\wedge \eta_i])\frac{\omega^n}{n!}\bigg)^{-\frac{3}{2}}\\
 &\cdot\bigg[2\Big(\int_X\tilde h(\eta_i,\eta_i)\frac{\omega^n}{n!}\Big)\Big(2\int_X h(\eta_i^{\dagger_{\tilde h}}\eta_i,G\eta_i^{\dagger_{\tilde h}}\eta_i)\frac{\omega^n}{n!}-\int_X\tilde h(\mathbf{d}^{\dagger_{\tilde h}}G[\eta_i\wedge \eta_i], \mathbf{d}^{\dagger_{\tilde h}}G[\eta_i\wedge \eta_i])\frac{\omega^n}{n!}\Big)\\
 &\ \ \ \ +\kappa\Big(-5\int_X\tilde h([G\eta_i^{\dagger_{\tilde h}}\eta_i,\mathbf{d}^{\dagger_{\tilde h}}G[\eta_i\wedge \eta_i]],\mathbf{d}^{\dagger_{\tilde h}}G[\eta_i\wedge \eta_i])\frac{\omega^n}{n!}\\
  &\ \ \ \ \ \ \ \ \ \ \ +5\int_Xh(G\mathbf{d}^{\dagger_{\tilde h}}\eta_{ i}^{\dagger_{\tilde h}}G[\eta_i\wedge \eta_i],\mathbf{d}^{\dagger_{\tilde h}}\eta_{ i}^{\dagger_{\tilde h}}G[\eta_i\wedge \eta_i])\frac{\omega^n}{n!}\\
 &\ \ \ \ \ \ \ \ \ \ \ -12\mathrm{Re}\int_X h(G\mathbf{d}^{\dagger_{\tilde h}}[\eta_i,G\eta_i^{\dagger_{\tilde h}}\eta_i],\mathbf{d}^{\dagger_{\tilde h}}\eta_{i}^{\dagger_{\tilde h}} G[\eta_i\wedge \eta_i])\frac{\omega^n}{n!}\\ &\ \ \ \ \ \ \ \ \ \ \ -4\int_X\tilde h(G[\eta_i\wedge\mathbf{d}G\eta_i^{\dagger_{\tilde h}}\eta_i],[\eta_i\wedge\mathbf{d}G\eta_i^{\dagger_{\tilde h}}\eta_i])\frac{\omega^n}{n!}\\
  &\ \ \ \ \ \ \ \ \ \ \ -\int_X\tilde h(G^2 [[\eta_i\wedge \eta_i]\wedge \eta_i],[[\eta_i\wedge \eta_i]\wedge \eta_i])\frac{\omega^n}{n!}\Big)\bigg]\\
  &+\kappa^2\bigg(\Big(\int_X\tilde h(\eta_i,\eta_i)\frac{\omega^n}{n!}\Big)^2+\kappa\int_X\tilde h(\mathbf{d}^{\dagger_{\tilde h}}G[\eta_i\wedge \eta_i], \mathbf{d}^{\dagger_{\tilde h}}G[\eta_i\wedge \eta_i])\frac{\omega^n}{n!}\bigg)^{-\frac{5}{2}}\\
 &\ \ \ \cdot \bigg|\int_X\tilde h( [(\mathbf{d}^{\dagger_{\tilde h}}G[\eta_i\wedge \eta_i])\wedge\eta_i],G[\eta_i\wedge \eta_i])\frac{\omega^n}{n!}+2\int_X\tilde h([\eta_i\wedge \mathbf{d} G\eta_i^{\dagger_{\tilde h}}\eta_i],G[\eta_i\wedge \eta_i])\frac{\omega^n}{n!}\bigg|^2.
\end{align*}
\end{theorem}

Applying this formula to the case when $X$ is a compact K\"ahler surface, we immediately have the following corollary.

\begin{corollary}
Suppose $X$ is a compact K\"{a}hler surface.
Let $\xi$ be a  local holomorphic vector field over a neighborhood $U\subseteq S$. Assume $\xi$ is relatively harmonic and $G(\mathbb{H}(\xi))^{\dagger_{\tilde{ h}}}\mathbb{H}(\xi)$ is a negative-definite operator on $\mathcal{A}_X^1(\End(E))$ at some point $s\in U$. If the holomorphic sectional curvature of $G^{\mathrm{WP}}$ is nonzero at $s$, then there exists a smooth and  strongly pseudo-convex Finsler metric $F$ such that $K_F(\xi)|_s>0$.
\end{corollary}

\section{Some Applications of Biswas--Schumacher's Curvature Formula}\label{sec2}

\subsection{Some Identities} 

\hspace*{5pt}

The following lemma collects some useful identities that will be used frequently later. The proof can be found in \cite{BS}, basically, they follow from the Ricci identity and the K\"ahler identity.

\begin{lemma}[\cite{BS}]\label{4m}
We have the following identities
\begin{enumerate}
\item $\mathbf{d}\nabla_i\eta_j+[\eta_i\wedge \eta_j]=0,$
\smallskip
\item $\mathbf{d}^{\dagger_{\tilde h}}\nabla_{\bar j}\eta_i+\eta_j^{\dagger_{\tilde h}}\eta_i=0,$
\smallskip
\item $\mathbf{d}^{\dagger_{\tilde h}}\nabla_i\eta_j=\mathbf{d}\nabla_{\bar i}\eta_j=0$,
\smallskip
\item $\nabla_{\bar j}\eta_i=\mathbf{d}R_{i\bar j},$
\smallskip
\item $\nabla_{i}\eta^{*_h}_{\bar j}=-\mathbf{d}^{*_h}R_{i\bar j}.$
\end{enumerate}
\end{lemma}

\subsection{Some Applications}

\hspace*{5pt}

Biswas--Schumacher's curvature formula \eqref{1} provides a differential-geometric tool to study the geometric properties of the base manifold. 

The following theorem collects some nice  results which exhibit the deep relations between the birational geometry and the positivity of holomorphic sectional curvature of manifolds.

\begin{theorem}[\cite{Yan,HW}]\label{yu}
\begin{enumerate}
  \item A compact Hermitian manifold with semipositive but not identically zero holomorphic sectional curvature has Kodaira dimension $-\infty$.
  \item A projective manifold with positive  holomorphic sectional curvature is uniruled.
\end{enumerate}
\end{theorem}

\begin{theorem}\label{thm2.4}
\begin{enumerate}
    \item If $(\mathcal{E},\Phi)$ is  an effectively parametrized holomorphic family of stable Higgs bundles  on a fixed  Riemann surface $X$ parametrized by a compact complex manifold $S$ such that $(\nabla_{\bar{i}}\eta_i)|_s\neq0$ for some $s\in S$, then the Kodaira dimension of $S$ is $-\infty$.
   \item If $(\mathcal{E},\Phi)$ is  an effectively parametrized holomorphic family of stable Higgs bundles of non-zero degree on a fixed  Riemann surface $X$ parametrized by a projective manifold $S$ such that $(\nabla_{\bar{i}}\eta_i)|_s\neq0$ for each $s\in S$, then $S$ is uniruled.
    \item If $(\mathcal{E},\Phi)$ is  an effectively parametrized holomorphic family of stable Higgs bundles  on a fixed  Riemann surface $X$ parametrized by  a compact K\"ahler manifold $S$, then either $S$ is a torus, or the Kodaira dimension of $S$ is $-\infty$. 
   \item If $\mathcal{E}$ is  an effectively parametrized holomorphic family of stable vector bundles with vanishing Chern classes on a fixed  compact K\"{a}hler manifold $X$ parametrized by a compact complex manifold $S$ such that $(\nabla_{\bar{i}}\eta_i)|_s\neq0$ for some $s\in S$, then the Kodaira dimension of $S$ is $-\infty$. 
\end{enumerate}
\end{theorem}

\begin{proof}
(1) If $X$ is a Riemann surface, the sign of  holomorphic sectional curvature of $G^{\mathrm{WP}}$ is the same as
$$
\mathcal{R}_{i\bar ii\bar i}=2\int_X\Tr(R_{i
\bar i}\Box R_{i\bar i})\omega_X,
$$
which is semipositive. To show (1), we show $\mathcal{R}_{i\bar{i}i\bar{i}} (1\leq i\leq m)$ can not be identically zero by contradiction.  If $\mathcal{R}_{i\bar ii\bar i}|_s=0$ at each point $s\in S$, then we have
\begin{align*}
 \bar\partial_{\mathcal{E}_s}R_{i\bar i}|_s=[\Phi_s,R_{i\bar i}|_s]=0.
\end{align*}
It follows from Lemma \ref{1} (4) that at each point $s\in S$, $(\nabla_{\bar{i}}\eta_i)|_s$ vanishes, which contradicts to our assumption. Then the conclusion follows from Theorem \ref{yu} (1).

(2)  From the proof of (1) and Theorem \ref{yu} (2), we can easily obtain the conclusion.

(3) This claim follows from the fact that a compact K\"ahler manifold with identically zero holomorphic sectional curvature is a torus \cite{Yan}. 

(4) By the assumption on vanishing of Chern classes, we have $R_{\alpha\bar \beta}|_s=0$ for any $1\leq\alpha,\beta\leq n$ and any $s\in S$. Then we have
   \begin{align*}
     (\bar\partial_{\mathcal{E}_s})^{\dagger_{\tilde{h}_s}}([R_{i\bar \beta},R_{i\bar \gamma}]d\bar z^\beta\wedge d\bar z^\gamma)|_s
     &=-2(g^{\bar \beta \alpha}\nabla_\alpha[R_{i\bar \beta},R_{i\bar \gamma}] d\bar z^\gamma)|_s\\
     &=-2(g^{\bar \beta \alpha}([\nabla_\alpha R_{i\bar \beta}, R_{i\bar \gamma}]+[R_{i\bar \beta},\nabla_\alpha R_{i\bar \gamma}]) d\bar z^\gamma)|_s\\
      &=-2(g^{\bar \beta \alpha}([\nabla_i R_{\alpha\bar \beta}, R_{i\bar \gamma}]+[R_{i\bar \beta},\nabla_i R_{\alpha\bar \gamma}]) d\bar z^\gamma)|_s\\
     &=0.
   \end{align*}
   The identity in Lemma \ref{1} (1) is locally written as \cite{BS}
   $$
   \bar\partial_{\mathcal{E}_s}((\nabla_iR_{i\bar \beta})d\bar z^\beta)|_s+([R_{i\bar \beta},R_{i\bar \gamma}]d\bar z^\beta\wedge d\bar z^\gamma)|_s=0,
   $$
  then the operator $(\bar\partial_{\mathcal{E}_s})^{\dagger_{\tilde{h}_s}}$ acting on both sides  gives rise to $(\bar\partial_{\mathcal{E}_s})^{\dagger_{\tilde{h}_s}}\bar\partial_{\mathcal{E}_s}((\nabla_iR_{i\bar \beta})d\bar z^\beta)|_s=0$, which implies
$$
\bar\partial_{\mathcal{E}_s}((\nabla_iR_{i\bar \beta})d\bar z^\beta)|_s=0,
$$  
namely $[\eta_i\wedge\eta_i]|_s=([R_{i\bar \beta},R_{i\bar \gamma}]d\bar z^\beta\wedge d\bar z^\gamma)|_s=0$.
   Therefore, again by Biswas--Schumacher's curvature formula, we have
$$
\mathcal{R}_{i\bar ii\bar i}=2\int_X\Tr(R_{i
\bar i}\Box R_{i\bar i})\omega_X\geq0.
$$
Then the conclusions also follow from the proof of (1). 
\end{proof}

Applying the above theorem to the moduli space of vector bundles on a fixed Riemann surface $X$, we have the following corollary.

\begin{corollary}\label{modulisp}
Assume $X$ is a Riemann surface  of genus $g\geq 1$, and let $r$ and $d$ be two positive integers such that they are coprime. Denote by $N_X(r,d)$ the coarse  moduli space  of semistable  vector bundles of rank $r$ and degree $d$ on $X$. Then $N_X(r,d)$ is either an abelian variety, or of Kodaira dimension $-\infty$.
\end{corollary}

\begin{proof}
 When $r$ and $d$ are coprime, it is known that $N_X(r,d)$  is a smooth projective variety and is a fine moduli space such that there exists a (global) universal vector bundle $\mathcal{E}$  over $X\times N_X(r,d)$. Then the conclusion follows from Theorem \ref{thm2.4} (3). 
\end{proof}

\begin{remark}
Assume $g\geq2, r\geq2$, let $L$ be a fixed line bundle of degree $d$ on $X$, and let $N_X(r,L)$ the coarse moduli space of semistable vector bundles of rank $r$ with determinant $L$ on $X$. It is known that $N_X(r,L)$ is uniruled (since $N_X(r,L)$ is simply-connected, the above corollary only implies it has Kodaira dimension $-\infty$), hence $N_X(r,d)$  contains a free rational curve $C$ whose normal bundle $N_C$ is semipositive. From the adjunction formula it follows that $-K_X$ is positive on the curve $C$, where $K_X$ denotes the canonical line bundle of $X$.  On the other hand,   by a result of Drezet and  Narasimhan \cite{DN}, the Picard group $\mathrm{Pic}(N_X(r,L))$ is isomorphic to $\mathbb{Z}$, which implies $-K_X$ is ample. Therefore, $N_X(r,L)$ is a Fano variety. By determinant map, $N_X(r,d)$ is a fibration over a Picard variety of deg $d$ with Fano fibers. However, this result does not immediately imply the Kodaira dimension of $N_X(r,d)$ is $-\infty$.
\end{remark}

In general, we guess Theorem \ref{thm2.4} (3) also holds for higher dimensional $X$, namely the following conjecture is proposed.

\begin{conjecture}
If $(\mathcal{E},\Phi)$ is  an effectively parametrized holomorphic family of stable Higgs bundles  on a fixed compact  K\"{a}hler manifold $X$ parametrized by a simply-connected compact complex manifold $S$, then the Kodaira dimension of $S$ is $-\infty$. 
\end{conjecture}

\section{Curvature on Direct Image Bundles}\label{sec3}

In this section, we will calculate the curvature formula on the locally free higher direct image sheaf $\mathbb{R}^d\pi_*\mathbf{H}_{\mathrm{r}}$ over $S\backslash Z^{(d)}$, where the notations for these objects can be found in Sect. \ref{sec1}. Our main aim of this section is to prove Theorem \ref{thm1}.

As a warm-up, we firstly  calculate the corresponding curvature tensor for $(\pi_*\mathbf{H}_\mathrm{r},H)$. Let $s_0$ be a point lying in  $S\backslash Z^{(0)}$, we choose a holomorphic trivialization $\{t_1,\cdots, t_{r_0}\}$ of $\pi_*\mathbf{H}_\mathrm{r}$ over an open  neighborhood $U\subseteq S\backslash Z^{(0)}$ containing $s_0$, and choose a normal  coordinate $\{s^1,\cdots,s^m\}$ for $U$ so that the $L^2$-metric $H$ satisfying
$$
\frac{\partial H_{a\bar b}}{\partial s^i}\Big|_{s_0}=0
$$
for any $1\leq i\leq m, 1\leq a,b\leq r_0$, where $H_{a\bar b}=H(t_a,t_b)$. Note that $\{t_1|_s,\cdots, t_{r_0}|_s\}$ forms a basis of $\mathbb{H}^0(X,\mathbf{H}_s)$. Therefore, the curvature tensor $\mathfrak{R}$ for $(\pi_*\mathbf{H}_\mathrm{r},H)$ at $s_0$ is given by
\begin{align*}
  \mathfrak{R}_{a\bar b i\bar j}|_{s_0}&=-\frac{\partial^2 H_{a\bar b}}{\partial s^i\partial\bar s^j}\Big|_{s_0}\\
  &=-\Big(\int_X h(\nabla_{\bar j}\nabla_it_a,t_b)\frac{\omega_X^n}{n!}\Big)\Big|_{s_0}-\Big(\int_X h(\nabla_it_a,\nabla_jt_b)\frac{\omega_X^n}{n!}\Big)\Big|_{s_0}.
\end{align*}

\begin{lemma}\label{vz}
We have the following equality at $s_0$
$$
\Big(\int_X h(\nabla_it_a,\nabla_jt_b)\frac{\omega_X^n}{n!}\Big)\Big|_{s_0}=\Big(\int_X \tilde h\big(G\big(\eta_i(t_a)\big),\eta_j(t_b)\big)\frac{\omega_X^n}{n!}\Big)\Big|_{s_0}.$$
\end{lemma}

\begin{proof}
Due to $\{t_a\}$ being  a holomorphic frame  over $U$, i.e. $ \nabla_{\bar j}t_a=0$ for any $1\leq a\leq r_0, 1\leq j\leq m$,
 fixing the indices $i,a$, we have
$$
\Big(\int_Xh((\nabla_it_a),t_b)\frac{\omega_X^n}{n!}\Big)\Big|_{s_0}=0
$$
for any $1\leq b\leq r_0$, which implies that the harmonic projection $P((\nabla_it_a)|_{s_0})$ of $(\nabla_it_a)|_{s_0}$ with respect to the Laplacian $\Box|_{s_0}$ on $(\mathcal{E}_{s_0},\Phi_{s_0},h_{s_0})$ vanishes.
It follows that
\begin{align*}
 \Big(\int_Xh(\nabla_it_a,\nabla_jt_b)\frac{\omega_X^n}{n!}\Big)\Big|_{s_0}&=\Big(\int_Xh(\Box G\nabla_it_a,\nabla_jt_b)\frac{\omega_X^n}{n!}\Big)\Big|_{s_0}\\
 &=\int_X\tilde h_{s_0}(\mathbf{d}|_{s_0} (G\nabla_it_a)|_{s_0},\mathbf{d}|_{s_0}(\nabla_jt_b)|_{s_0})\frac{\omega_X^n}{n!}\\
 &\ \ \ \ +\int_X\tilde h_{s_0}(\mathbf{d}^{\dagger_{\tilde{h}}}|_{s_0} (G\nabla_it_a)|_{s_0},\mathbf{d}^{\dagger_{\tilde{h}}}|_{s_0}(\nabla_jt_b)|_{s_0})\frac{\omega_X^n}{n!}\\
 &=\int_X\tilde h_{s_0}(G|_{s_0}\mathbf{d}|_{s_0}( \nabla_it_a)|_{s_0},\mathbf{d}|_{s_0}(\nabla_jt_b)|_{s_0})\frac{\omega_X^n}{n!}.
\end{align*}
By the Ricci identity, we calculate
\begin{align*}
 \mathbf{d}|_{s_0} (\nabla_it_a)|_{s_0}=&-(\nabla_i(\Phi(t_a)))|_{s_0}-(R_{i\bar \alpha}d\bar z_\alpha (t_a))|_{s_0}+\Phi_{s_0}((\nabla_it_a)|_{s_0})\\
 =&-\eta_i(t_a)|_{s_0},
\end{align*}
which yields  that
$$
\Big(\int_X h(\nabla_it_a,\nabla_jt_b)\frac{\omega_X^n}{n!}\Big)\Big|_{s_0}=\Big(\int_X\tilde h\big(G\big(\eta_i (t_a)\big),\eta_j( t_b)\big)\frac{\omega_X^n}{n!}\Big)\Big|_{s_0}.
$$
 The desired equality is obtained.
\end{proof}

\begin{lemma}\label{vc}
We have the following equality at $s_0$
  \begin{align*}
  \Big(\int_X h(\nabla_{\bar j}\nabla_it_a,t_b)\frac{\omega_X^n}{n!}\Big)\Big|_{s_0}=&-\bigg(\frac{1}{r\mathrm{Vol}(X,\omega_X)}\int_X\Tr({R}_{i\bar j}|_{s_0})\frac{\omega_X^n}{n!}\bigg)(H_{a\bar b}|_{s_0})\\
 &+\Big(\int_X h\big(G\big((\eta_j^{\dagger_{\tilde h}}\eta_i)(t_a)\big),t_b\big)\frac{\omega_X^n}{n!}\Big)\Big|_{s_0}.
  \end{align*}
\end{lemma}

\begin{proof} 
Again by the holomorphicity of $t_a$ and  by the Ricci identity, we have
$$
(\nabla_{\bar j}\nabla_it_a)|_{s_0}=-{R}_{i\bar j}(t_a)|_{s_0}.
$$
Since $(\mathcal{E}_{s},\Phi_{s})$ is a stable Higgs bundle, the harmonic projection $P({R}_{i\bar j}|_{s})$ of $R_{i\bar j}|_{s}$ with respect to the induced Laplacian $\Box|_{s}$ on $\End(\mathcal{E}_s)$ must satisfy
$$
P({R}_{i\bar j}|_{s})=c_{i\bar j}|_s\cdot\mathrm{Id}_{\mathcal{E}_s}
$$
where  $c_{i\bar j}ds^i\wedge d\bar s^j$ is a (1,1)-form over $S$ given by
$$
c_{i\bar j}=\frac{1}{r\mathrm{Vol}(X,\omega_X)}\int_X\Tr(R_{i\bar j})\frac{\omega_X^n}{n!}.
$$
 Therefore, we arrive at
\begin{align*}
  \Big(\int_X h(\nabla_{\bar j}\nabla_it_a,t_b)\frac{\omega_X^n}{n!}\Big)\Big|_{s_0}&=-\Big(\int_X h({R}_{i\bar j}(t_a),t_b)\frac{\omega_X^n}{n!}\Big)\Big|_{s_0}\\
  &=-\int_X h_{s_0}([P({R}_{i\bar j}|_{s_0})+(G\Box{R}_{i\bar j})|_{s_0}](t_a|_{s_0}),t_b|_{s_0})\frac{\omega_X^n}{n!}\\
  &=-(c_{i\bar j}H_{a\bar b})|_{s_0}+\Big(\int_X h\big(G\big((\eta_j^{\dagger_{\tilde h}}\eta_i)(t_a)\big),t_b\big)\frac{\omega_X^n}{n!}\Big)\Big|_{s_0},
\end{align*}
where we apply  the identity in Lemma \ref{1} (4) for the third equality.
\end{proof}

Combining the above two lemmas together leads to the following theorem as an analog of Theorem 1 in \cite{TW}

\begin{theorem} \label{rt} 
The curvature tensor $\mathfrak{R}$ of $(\pi_*\mathbf{H}_\mathrm{r},H)$ over $S\backslash Z^{(0)}$ is given by
\begin{align*}
 \mathfrak{R}_{a\bar b i\bar j}=\ &\bigg(\frac{1}{r\mathrm{Vol}(X,\omega_X)}\int_X\Tr(R_{i\bar j})\frac{\omega_X^n}{n!}\bigg)\cdot H_{a\bar b}-\int_X h\big(G\big((\eta_j^{\dagger_{\tilde h}}\eta_i)(t_a)\big),t_b\big)\frac{\omega_X^n}{n!}\\
 &-\int_X \tilde h\big(G(\eta_i(t_a)),\eta_j(t_b)\big)\frac{\omega_X^n}{n!}.
\end{align*}
\end{theorem}

Next we calculate the curvature tensor for higher-order direct image sheaf  $\mathbb{R}^d\pi_*\mathbf{H}_\mathrm{r}$. As before, for a point $s_0\in S\backslash Z^{(d)}$, we choose a holomorphic trivialization $\{t_1,\cdots, t_{r_d}\}$ of $\mathbb{R}^d\pi_*\mathbf{H}_\mathrm{r}$ over an open  neighborhood $U\subseteq S\backslash Z^{(d)}$ , where each $t_a$ lies in $\mathcal{A}^d_{X\times U}(\mathcal{E})$ and satisfies $\mathbf{d}t_a=0$. A key observation similar to Lemma 2 in \cite{Sch} and Lemma 2.1 in \cite{m2}
is that for $\chi\in \mathcal{A}^d_{X\times U}(\mathcal{E})$ satisfying  $\mathbf{d}\chi=0$ there exists a form $\theta\in \mathcal{A}^{d-1}_{X\times V}(\mathcal{E})$ on some open neighborhood $V\subseteq U$ of $s_0$ such that $(\chi+\mathbf{d}\theta)|_{s}=P(\chi|_s)$ for every $s\in V$. As a consequence, we can assume the restriction $t_a|_{s}$
is harmonic for any $s\in U$ and $1\leq a\leq r_d$. The proof of Lemma \ref{vz} and Lemma \ref{vc} continue to wok  for the case of  higher-order direct image sheaf by minor modifications, therefore, we have  the following theorem as an analog of Theorem 1 in \cite{m2}.

\begin{theorem}\label{curfor}
The curvature tensor $\mathfrak{R}$ of $(\mathbb{R}^d\pi_*\mathbf{H}_\mathrm{r},H)$ over $S\backslash Z^{(d)}$ is given by
\begin{align*}
 \mathfrak{R}_{a\bar b i\bar j}=\ &\bigg(\frac{1}{r\mathrm{Vol}(X,\omega_X)}\int_X\Tr(R_{i\bar j})\frac{\omega_X^n}{n!}\bigg)\cdot H_{a\bar b}+\int_X \tilde h\big(G\big(\eta_j^{\dagger_{\tilde h}}(t_a)\big),\eta_i^{\dagger_{\tilde h}}(t_b)\big)\frac{\omega_X^n}{n!}\\
 &-\int_X \tilde{h}\big(G\big((\eta_j^{\dagger_{\tilde h}}\eta_i)(t_a)\big),t_b\big)\frac{\omega_X^n}{n!}-\int_X \tilde h\big(G(\eta_i\wedge t_a),\eta_j\wedge t_b\big)\frac{\omega_X^n}{n!}.
\end{align*}
\end{theorem}

\begin{proof}
Since $\mathbf{d}t_a=0$, the same  arguments  as  in Lemma 3.1 of \cite{m2} show that $$(\nabla_{\bar i}t_a)|_{s_0}=\mathbf{d}|_{s_0}\delta_{\bar i}(t_a)|_{s_0},$$
where $\delta_{\bar i}(t_a)\in \mathcal{A}^{d-1}_{X\times U}(\mathcal{E})$. Hence,  fixing the indices $i,a$, we also have
$$
\Big(\int_X\tilde h(\nabla_it_a,t_b)\frac{\omega_X^n}{n!}\Big)\Big|_{s_0}=0
$$
for any $1\leq b\leq r_d$. In addition, we have
$$
\mathbf{d}^{\dagger_{\tilde{h}}}|_{s_0}(\nabla_j t_b)|_{s_0}=(\nabla_j\mathbf{d}^{\dagger_{\tilde h}} t_b)|_{s_0}=0.
$$
Therefore, the following identity as  in Lemma \ref{vz} holds
$$
\Big(\int_X \tilde h(\nabla_it_a,\nabla_jt_b)\frac{\omega_X^n}{n!}\Big)\Big|_{s_0}=\Big(\int_X \tilde h\big(G(\eta_i\wedge t_a),\eta_j\wedge t_b\big)\frac{\omega_X^n}{n!}\Big)\Big|_{s_0}.
$$
To obtain the analog of Lemma \ref{vc},  we need to calculate
$(\nabla_{\bar j}\nabla_it_a)|_{s_0}= (\nabla_i\nabla_{\bar j}t_a)|_{s_0}-({R}_{i\bar j}(t_a))|_{s_0}$.
The first term of the right hand side can be  computed as
\begin{align*}
  (\nabla_i\nabla_{\bar j}t_a)|_{s_0}&=(\nabla_i\mathbf{d}\delta_{\bar j}(t_a))|_{s_0}\\
  &=\mathbf{d}|_{s_0}(\nabla_i\delta_{\bar j}(t_a))|_{s_0}+R_{i\bar \beta}|_{s_0}d\bar z^\beta\wedge \delta_{\bar j}(t_a)|_{s_0}+\Phi_{s_0}\wedge \delta_{\bar j}(t_a)|_{s_0}\\
  &=\mathbf{d}|_{s_0}(\nabla_i\delta_{\bar j}(t_a))|_{s_0}+\eta_i|_{s_0}\wedge \delta_{\bar j}(t_a)|_{s_0}.
\end{align*}
Then we have
\begin{align*}
 \Big(\int_X \tilde h(\nabla_{\bar j}\nabla_it_a,t_b)\frac{\omega_X^n}{n!}\Big)\Big|_{s_0}&=\Big(\int_X h(\eta_i\wedge\delta_{\bar j}(t_a),t_b)\frac{\omega_X^n}{n!}\Big)\Big|_{s_0}
 -\Big(\int_X \tilde h(R_{i\bar j}(t_a),t_b)\frac{\omega_X^n}{n!}\Big)\Big|_{s_0}\\
 &=\Big(\int_X h(\delta_{\bar j}(t_a),\eta_i^{\dagger_{\tilde h}}(t_b))\frac{\omega_X^n}{n!}\Big)\Big|_{s_0}
 -\Big(\int_X \tilde h(R_{i\bar j}(t_a),t_b)\frac{\omega_X^n}{n!}\Big)\Big|_{s_0}
\end{align*}
It follows from  the harmonicity of $t_a|_{s_0}$,  the Ricci identity and the K\"{a}hler identity that
\begin{align*}
\mathbf{d}^{\dagger_{\tilde{h}}}|_{s_0}(\nabla_{\bar i} t_b)|_{s_0}&=-\sqrt{-1}[\Lambda_{\omega_X},\partial_{\mathcal{E}_{s_0}}+\Phi^{*_h}|_{s_0}](\nabla_{\bar i} t_b)|_{s_0}\\
&=-\sqrt{-1}[\Lambda_{\omega_X},-R_{\bar i\beta }|_{s_0}d z^\beta-(\nabla_{\bar i}\Phi^{*_h})|_{s_0}]t_b|_{s_0}\\
&=-\eta_i^{\dagger_{\tilde h}}(t_b)|_{s_0},
\end{align*}
which implies that
\begin{align*}
  \Big(\int_X \tilde h(\nabla_{\bar j}\nabla_it_a,t_b)\frac{\omega_X^n}{n!}\Big)\Big|_{s_0}&=-\Big(\int_X \tilde h(\nabla_{\bar j}t_a,\nabla_{\bar i}t_b)\frac{\omega_X^n}{n!}\Big)\Big|_{s_0}-\Big(\int_X \tilde h(R_{i\bar j}(t_a),t_b)\frac{\omega_X^n}{n!}\Big)\Big|_{s_0}.
\end{align*}
Moreover, since $(\nabla_{\bar j}t_a)|_{s_0}$ is $\mathbf{d}|_{s_0}$-exact, we have
\begin{align*}
  \Big(\int_X \tilde h(\nabla_{\bar j}t_a,\nabla_{\bar i}t_b)\frac{\omega_X^n}{n!}\Big)\Big|_{s_0}&=\int_X \tilde h(\mathbf{d}|_{s_0}G|_{s_0}\mathbf{d}^{\dagger_{\tilde{h}}}|_{s_0}(\nabla_{\bar j}t_a)|_{s_0},(\nabla_{\bar i}t_b)|_{s_0})\frac{\omega_X^n}{n!}\\
  &=\Big(\int_X \tilde h(G(\eta_j^{\dagger_{\tilde h}}(t_a)),\eta_i^{\dagger_{\tilde h}}(t_b))\frac{\omega_X^n}{n!}\Big)\Big|_{s_0},
\end{align*}
which is a term automatic vanishing in the case of zero-order.  According to above calculations, we finally get the theorem.
\end{proof}

\section{Curvature of Finsler Metric}\label{sec4}

 In this section we introduce a  Finsler metric on $S$ and calculate the corresponding curvature. Our main aim of this section is to prove Theorem \ref{thm2}. Firstly, we recall some basic definitions of  Finsler metric.
 
 \begin{definition}[\cite{Kob,TY}]
 Let $S$ be an $m$-dimensional complex manifold, and $\pi:TS\rightarrow S$ be the   holomorphic tangent bundle  of $S$.   The zero section of $TS$ is denoted by $o(S)$. A point  $v$ lying in  $TS$ is written as $v=(s,\xi)$ in terms of a local trivialization of $TS$ with $s=(s^1,\cdots,s^m), \xi=\sum_{i=1}^m\xi^i\frac{\partial}{\partial s^i}$. The vertical bundle $\mathcal{V}$ on $TS\backslash o(S)$ is defined by $\mathcal{V}=\Ker(d\pi: TTS\rightarrow TS)|_{TS\backslash o(S)}$, whose local frame at $v\in TS\backslash o(S)$ is given by $\{\frac{\partial}{\partial\xi^1},\cdots, \frac{\partial}{\partial\xi^m}\}$.
 \begin{enumerate}
 \item A continuous real valued function $F: TS\rightarrow \mathbb{R}$  is called a \emph{complex Finsler metric} on $S$ if the following conditions are satisfied:
    \begin{itemize}
      \item $F(v)\geq0$ for any $v\in TS$, and $F(v)=0$ if and only if $v\in o(S)$,
      \item $F(c \cdot v)=|c|F(v)$ for any $c\in \mathbb{C}^*$, where $c \cdot v=(s, c\xi)$.
    \end{itemize}
 \item For a Finsler metric $F$, the Hermitian matrix 
    $$
      G_{i\bar j}=\frac{\partial^2F^2}{\partial \xi^i\partial\bar \xi^j}
      $$ 
     is called the \emph{Levi matrix} associated to $F$.
 \item A Finsler metric $F$ is called \emph{smooth} if for any open subset $U\subset S$ and any nontrivial $C^\infty$-section $u$ of $TS$, $F(u)$ is a smooth function on $U$.
 \item A Finsler metric $F$ is called \emph{strongly pseudo-convex} if the Levi matrix defines a Hermitiam metric on the vertical bundle $\mathcal{V}$, namely  for any nonzero $W=W^i\frac{\partial}{\partial \xi^i}\in \mathcal{V}_v$, we have
  $$
     G_{i\bar j}(v)W^i\bar W^j>0.
     $$
 \item $S$ is called a \emph{complex Finsler manifold} if it is equipped with a complex Finsler metric that is smooth and strongly pseudo-convex.
 \item Given a smooth Finsler metric $F$ on $S$, for any holomorphic map $f:\mathfrak{D}\rightarrow S$ from the unit disk $\mathfrak{D}$ to $S$, there is  an induced  Hermitian metric $f^*F^2$ on $\mathfrak{D}$, and the corresponding Gauss curvature $K(f,F)$ is given by
 \begin{align*}
  K(f,F)=-\frac{2}{f^*F^2}\frac{\partial^2\log  f^*F^2}{\partial t\partial\bar t},
 \end{align*}
where $\{t\}$ is the local holomorphic coordinate on $\mathfrak{D}$.  For a point $v\in TS$, let $\mathcal{D}_v$ be the set consisting of  holomorphic maps $f:\mathfrak{D}\rightarrow S$ such that $f(0)=s$ and $f'(0)=cv$ for some nonzero constant $c$,  then we introduce the \emph{holomorphic sectional curvature} of the Finsler metric $F_\kappa$ in the direction $v$ as
 $$
 K_F(v)=\sup_{f\in\mathcal{D}_v}K(f,F)|_0.
 $$
 \end{enumerate}
 \end{definition}

Return to our effectively parametrized  family $(\mathcal{E},\Phi)$ over $S$. $S$ is endowed with a Finsler metric $F_\kappa$ introduced in Sect. \ref{sec1}. 

\begin{proposition}\label{spc}
The Finsler metric $F_\kappa$ is smooth and  strongly pseudo-convex, thus $S$ is a complex Finsler manifold.
\end{proposition}

\begin{proof} 
The Levi matrix associated to $F_\kappa$ is given by
\begin{align*}
  G_{i\bar j}(v)=&\ \frac{(\int_X \tilde h(\mathbb{H}(\xi),\mathbb{H}(\xi))\frac{\omega^n}{n!})(\int_X \tilde h(\eta_i,\eta_j)\frac{\omega^n}{n!})+(\int_X \tilde h(\eta_i,\mathbb{H}(\xi))\frac{\omega^n}{n!})(\int_X \tilde h(\mathbb{H}(\xi),\eta_j)\frac{\omega^n}{n!})}{F^2_\kappa(v)}\\
  &-\frac{(\int_X \tilde h(\mathbb{H}(\xi),\mathbb{H}(\xi))\frac{\omega^n}{n!})^2(\int_X \tilde h(\eta_i,\mathbb{H}(\xi))\frac{\omega^n}{n!})(\int_X \tilde h(\mathbb{H}(\xi),\eta_j)\frac{\omega^n}{n!})}{F^6_\kappa(v)}\\
  &+2\kappa\frac{\int_X\tilde h( [ \eta_i\wedge\mathbb{ H}(\xi)],G[\eta_j\wedge \mathbb{H}(\xi)])\frac{\omega^n}{n!}}{F^2_\kappa(v)}\\
  &-\kappa\frac{(\int_X \tilde h(\mathbb{H}(\xi),\mathbb{H}(\xi))\frac{\omega^n}{n!})(\int_X \tilde h([\mathbb{H}(\xi)\wedge \eta_i],G[\mathbb{H}(\xi)\wedge  \mathbb{H}(\xi)] )\frac{\omega^n}{n!})(\int_X \tilde h(\mathbb{H}(\xi),\eta_j)\frac{\omega^n}{n!})}{F^6_\kappa(v)}\\
   &-\kappa\frac{(\int_X \tilde h(\mathbb{H}(\xi),\mathbb{H}(\xi))\frac{\omega^n}{n!})(\int_X \tilde h(\eta_i,\mathbb{H}(\xi))\frac{\omega^n}{n!})(\int_X \tilde h(G[\mathbb{H}(\xi)\wedge  \mathbb{H}(\xi)],[\mathbb{H}(\xi)\wedge \eta_j] )\frac{\omega^n}{n!})}{F^6_\kappa(v)}\\
  &-\kappa^2\frac{(\int_X\tilde h([\eta_i\wedge \mathbb{H}(\xi),G [ \mathbb{H}(\xi)\wedge\mathbb{ H}(\xi)])\frac{\omega^n}{n!})(\int_X\tilde h( G[ \mathbb{H}(\xi)\wedge\mathbb{ H}(\xi)],[\eta_j\wedge \mathbb{H}(\xi)])\frac{\omega^n}{n!})}{F^6_\kappa(v)}.
\end{align*}
We only need to show $ G_{i\bar i}(v)>0$ for any  $v\in TS\backslash o(S)$.

Cauchy--Schwarz inequality provides us
\begin{align*}
 &\ \Big(\int_X\tilde h([\eta_i\wedge \mathbb{H}(\xi),G [ \mathbb{H}(\xi)\wedge\mathbb{ H}(\xi)])\frac{\omega^n}{n!}\Big)\Big(\int_X\tilde h( G[ \mathbb{H}(\xi)\wedge\mathbb{ H}(\xi)],[\eta_i\wedge \mathbb{H}(\xi)])\frac{\omega^n}{n!}\Big)\\
 =&\ \Big(\int_X\tilde h(\mathbf{d}^{\dagger_{\tilde h}}G[\eta_i\wedge \mathbb{H}(\xi)],\mathbf{d}^{\dagger_{\tilde h}}G [ \mathbb{H}(\xi)\wedge\mathbb{ H}(\xi)])\frac{\omega^n}{n!}\Big)\Big(\int_X\tilde h( \mathbf{d}^{\dagger_{\tilde h}}G[ \mathbb{H}(\xi)\wedge\mathbb{ H}(\xi)],\mathbf{d}^{\dagger_{\tilde h}}G[\eta_i\wedge \mathbb{H}(\xi)])\frac{\omega^n}{n!}\Big)\\
 \leq &\ \Big(\int_X\tilde h( [ \mathbb{H}(\xi)\wedge\mathbb{ H}(\xi)],G[\mathbb{H}(\xi)\wedge \mathbb{H}(\xi)])\frac{\omega^n}{n!}\Big)\Big(\int_X\tilde h([\eta_i\wedge\mathbb{ H}(\xi)],G[\eta_i\wedge \mathbb{H}(\xi)])\frac{\omega^n}{n!}\Big),
\end{align*}
 and similarly
\begin{align*}
 &\  \Big|\Big(\int_X \tilde h([\mathbb{H}(\xi)\wedge \eta_i],G[\mathbb{H}(\xi)\wedge  \mathbb{H}(\xi)])\frac{\omega^n}{n!}\Big)\Big(\int_X \tilde h(\mathbb{H}(\xi),\eta_i)\frac{\omega^n}{n!}\Big)\Big|\\
 \leq &\ \Big(\int_X \tilde h(\mathbb{H}(\xi),\mathbb{H}(\xi))\frac{\omega^n}{n!}\Big)^{\frac{1}{2}}\Big(\int_X \tilde h(\eta_i,\eta_i)\frac{\omega^n}{n!}\Big)^{\frac{1}{2}}\Big(\int_X \tilde h([\mathbb{H}(\xi)\wedge \mathbb{H}(\xi)],G[\mathbb{H}(\xi)\wedge \mathbb{H}(\xi)])\frac{\omega^n}{n!}\Big)^{\frac{1}{2}}\\
 & \cdot\Big(\int_X\tilde h( [ \eta_i\wedge\mathbb{ H}(\xi)],G[\eta_i\wedge \mathbb{H}(\xi)])\frac{\omega^n}{n!}\Big)^{\frac{1}{2}}.
\end{align*}

Then we have
\begin{align*}
  G_{i\bar i}(v)\ \ & \\
 \geq  \frac{1}{F^6_\kappa(v)}\bigg[& \Big(\int_X \tilde h(\mathbb{H}(\xi),\mathbb{H}(\xi))\frac{\omega^n}{n!}\Big)^3\Big(\int_X \tilde h(\eta_i,\eta_i)\frac{\omega^n}{n!}\Big)\\
  &+\kappa\bigg(\Big(\int_X \tilde h(\mathbb{H}(\xi),\mathbb{H}(\xi))\frac{\omega^n}{n!}\Big)\Big(\int_X\tilde h( [ \eta_i\wedge\mathbb{ H}(\xi)],G[\eta_i\wedge \mathbb{H}(\xi)])\frac{\omega^n}{n!}\Big)^{\frac{1}{2}}\\
  &\ \ \ \ \ \ \ -\Big(\int_X \tilde h(\mathbb{H}(\xi),\mathbb{H}(\xi))\frac{\omega^n}{n!}\Big)^{\frac{1}{2}}\Big(\int_X \tilde h(\eta_i,\eta_i)\frac{\omega^n}{n!}\Big)^{\frac{1}{2}}\Big(\int_X \tilde h([\mathbb{H}(\xi)\wedge \mathbb{H}(\xi)],G[\mathbb{H}(\xi)\wedge  \mathbb{H}(\xi)] )\frac{\omega^n}{n!}\Big)^{\frac{1}{2}}\bigg)^2\\
  &+\kappa^2\Big(\int_X\tilde h( [ \mathbb{H}(\xi)\wedge\mathbb{H}(\xi)],G[\mathbb{H}(\xi)\wedge \mathbb{H}(\xi)])\frac{\omega^n}{n!}\Big)\Big(\int_X\tilde h([\eta_i\wedge\mathbb{ H}(\xi)],G[\eta_i\wedge \mathbb{H}(\xi)])\frac{\omega^n}{n!}\Big)\bigg],
\end{align*}
which is positive for any  $v\in TS\backslash o(S)$.
\end{proof}

\begin{corollary}[{\cite[Proposition 2.5]{Wan}}]
The holomorphic sectional curvature of the Finsler metric $F_\kappa$ is given by
\begin{align*}
 K_{F_\kappa}(v)=2 \mathbb{R}_{i
 \bar jk\bar l}\frac{\xi^i\bar\xi^j\xi^k\bar\xi^l}{F_\kappa^2},
\end{align*}
where
$$\mathbb{R}_{i
 \bar jk\bar l}=-\frac{\partial^2 G_{i\bar j}}{\partial s^k\partial \bar s^l}+G^{\bar qp}\frac{G_{i\bar q}}{\partial s^k}\frac{G_{p\bar j}}{\partial\bar s^l}.$$
\end{corollary}

Let $\{t\}$ be the local holomorphic coordinate on $\mathfrak{D}$, then  for the holomorphic maps $f:\mathfrak{D}\rightarrow S$, we have
\begin{align*}
 f^*F_{\kappa}^2(\eta_f)&=\sqrt{\Big(\int_X\tilde h(\eta_f,\eta_f)\frac{\omega^n}{n!}\Big)^2+\kappa\int_X\tilde h(\nabla_f\eta_f, \nabla_f\eta_f)\frac{\omega^n}{n!}-\kappa\int_X\tilde h(P\nabla_f\eta_f, P\nabla_f\eta_f)\frac{\omega^n}{n!}}
\end{align*}
where $\eta_f=\mathbb{H}(f_*(\frac{\partial}{\partial t}))$, $\nabla_f$ is  the covariant derivative along the vector field $f_*(\frac{\partial}{\partial t})$, and $P$ is fiberwise projection on harmonic component.
Let $\{s^1,\cdots,s^m\}$ be the normal coordinate at $s_0\in S$ with respect to the Weil--Petersson-type metric $G^{\mathrm{WP}}$ on $S$.
To derive the the  holomorphic sectional curvature of the Finsler metric $F_\kappa$, we are interested in the computations of the following terms
\begin{align*}
A=&\ \bigg(\int_X\tilde h(\nabla_i\eta_i,\nabla_{\bar i}\nabla_i^2\eta_i)\frac{\omega^n}{n!}+\int_X\tilde h(\nabla_i\nabla_{\bar i}\nabla_i\eta_i,\nabla_i\eta_i)\frac{\omega^n}{n!}\bigg)\bigg|_{s_0},\\
B=&\ \bigg(\int_X\tilde h(\nabla_i^2\eta_i,\nabla_i^2\eta_i)\frac{\omega^n}{n!}\bigg)\bigg|_{s_0},\\
C=&\ \bigg(\int_X\tilde h(\nabla_{\bar i}\nabla_i\eta_i,\nabla_{\bar i}\nabla_i\eta_i)\frac{\omega^n}{n!}\bigg)\bigg|_{s_0},\\
D=&\ \bigg(\int_X\tilde h(\nabla_i(P\nabla_i \eta_i),\nabla_i(P\nabla_i \eta_i))\frac{\omega^n}{n!}+\int_X\tilde h(\nabla_{\bar i}(P\nabla_i \eta_i),\nabla_{\bar i}(P\nabla_i \eta_i))\frac{\omega^n}{n!}\bigg)\bigg|_{s_0},\\
E=&\ \bigg(\int_X\tilde h(\nabla_i^2\eta_i,\nabla_i\eta_i)\frac{\omega^n}{n!}+\int_X\tilde h(\nabla_i\eta_i,\nabla_{\bar i}\nabla_i\eta_i)\frac{\omega^n}{n!}\bigg)\bigg|_{s_0}\\
&\ \cdot\bigg(\int_X\tilde h(\nabla_{\bar i}\nabla_i\eta_i,\nabla_i\eta_i)\frac{\omega^n}{n!}+\int_X\tilde h(\nabla_i\eta_i,\nabla^2_i\eta_i)\frac{\omega^n}{n!}\bigg)\bigg|_{s_0}.
\end{align*}
The key point is to express them in terms of Kodaira--Spencer image.

\begin{lemma}\label{lo} 
The following identities hold:
\begingroup
\allowdisplaybreaks
\begin{align*}
(1) \ \ A&=\bigg(-6 \int_Xh(G\mathbf{d}^{\dagger_{\tilde h}}\eta_{ i}^{\dagger_{\tilde h}}G[\eta_i\wedge \eta_i],   \mathbf{d}^{\dagger_{\tilde h}}\eta_{ i}^{\dagger_{\tilde h}}G[\eta_i\wedge \eta_i])\frac{\omega^n}{n!}\\&\ \ \ \ \ \  +12\mathrm{Re}\int_X h(G\mathbf{d}^{\dagger_{\tilde h}}[\eta_i,G\eta_i^{\dagger_{\tilde h}}\eta_i],\mathbf{d}^{\dagger_{\tilde h}}\eta_{i}^{\dagger_{\tilde h}} G[\eta_i\wedge \eta_i])\frac{\omega^n}{n!}\qquad\qquad\qquad\qquad\qquad\qquad \\
  &\ \ \ \ \  \  +5\int_X\tilde h([G(\eta_i^{\dagger_{\tilde h}}\eta_i),\mathbf{d}^{\dagger_{\tilde h}}G[\eta_i\wedge \eta_i]],\mathbf{d}^{\dagger_{\tilde h}}G[\eta_i\wedge \eta_i])\frac{\omega^n}{n!}\bigg)\bigg|_{s_0},\\
(2) \ \ B&=\bigg(\int_X\tilde h(G^2 [[\eta_i\wedge \eta_i]\wedge \eta_i],[[\eta_i\wedge \eta_i]\wedge \eta_i])\frac{\omega^n}{n!}\bigg)\bigg|_{s_0},\\
  (3) \ \ C&= \bigg(4\int_X\tilde h(P[\eta_i,G\eta_i^{\dagger_{\tilde h}}\eta_i],P[\eta_i,G\eta_i^{\dagger_{\tilde h}}\eta_i])\frac{\omega^n}{n!}\\
&\ \ \ \ \ \
 +4\int_X\tilde h(G[\eta_i\wedge\mathbf{d}G\eta_i^{\dagger_{\tilde h}}\eta_i],[\eta_i\wedge\mathbf{d}G\eta_i^{\dagger_{\tilde h}}\eta_i])\frac{\omega^n}{n!}\\
&\ \ \ \ \  \  +\int_X h(G\mathbf{d}^{\dagger_{\tilde h}}\eta_i^{\dagger_{\tilde h}}G[\eta_i\wedge \eta_i],\mathbf{d}^{\dagger_{\tilde h}}\eta_i^{\dagger_{\tilde h}}G[\eta_i\wedge \eta_i])\frac{\omega^n}{n!}\bigg)\bigg|_{s_0},\\
 (4) \ \  D&= \bigg(4\int_X\tilde h(P[\eta_i,G\eta_i^{\dagger_{\tilde h}}\eta_i
],P[\eta_i,G\eta_i^{\dagger_{\tilde h}}\eta_i
])\frac{\omega^n}{n!}\bigg)\bigg|_{s_0},\\
 (5) \ \ E&=\bigg|\Big(\int_X\tilde h( [(\mathbf{d}^{\dagger_{\tilde h}}G[\eta_i\wedge \eta_i])\wedge \eta_i],G[\eta_i\wedge \eta_i])\frac{\omega^n}{n!} \\
 & \ \ \ \ \ \ +2\int_X\tilde h([\eta_i\wedge \mathbf{d}G\eta_i^{\dagger_{\tilde h}}\eta_i],G[\eta_i\wedge \eta_i])\frac{\omega^n}{n!}\Big)\Big|_{s_0}\bigg|^2.
\end{align*}
\endgroup
\end{lemma}

\begin{proof}
(1) Firstly, we introduce the following notations. Let $A=\alpha+\beta$ with $\alpha\in \mathcal{A}_X^{1,0}(\End(E))$ and $\beta\in \mathcal{A}_X^{0,1}(\End(E))$, then we define
\begin{align*}
\Lambda\cdot [A^{*_h}\wedge A]=-\sqrt{-1}\Lambda_\omega[\alpha^{*_h}\wedge\alpha]+\sqrt{-1}\Lambda_\omega[\beta^{*_h}\wedge\beta].
\end{align*}
It is clear that $\Lambda\cdot [A^{*_h}\wedge A]=A^{\dagger_{\tilde h}}A$ by the K\"{a}hler identities.

For our purpose, we calculate
\begin{align*}
\eta_i^{\dagger_{\tilde h}}\nabla_i\eta_i&=\Lambda\cdot [\eta^{*_h}_{\bar i}\wedge \nabla_i\eta_i]\\
&=\Lambda\cdot\nabla_i[\eta^{*_h}_{\bar i}\wedge \eta_i]-\Lambda\cdot[\nabla_i\eta^{*_h}_{\bar i}\wedge\eta_i]\\
&=-\Box\nabla_iR_{i\bar i}-\mathbf{d}^{\dagger_{\tilde h}}[\eta_i,R_{i\bar i}]+\Lambda\cdot[\mathbf{d}^{*_h}R_{i\bar i}\wedge\eta_i]\\
&=-\Box\nabla_iR_{i\bar i}-2\mathbf{d}^{\dagger_{\tilde h}}[\eta_i,R_{i\bar i}].
 \end{align*}
Then, together with  the Ricci identity, we have
\begin{align*}
 & \int_X\tilde h(\nabla_{\bar i}\nabla^2_i\eta_i,\nabla_i\eta_i)\frac{\omega^n}{n!}\\
  =\ &\int_X\tilde h(3\nabla_i[\eta_i,R_{i\bar i}] +\mathbf{d} \nabla^2_iR_{i\bar i},\nabla_i\eta_i)\frac{\omega^n}{n!}\\
  =\ &-3\int_X\tilde h([R_{i\bar i},\nabla_i\eta_i],\nabla_i\eta_i)\frac{\omega^n}{n!}+3\int_X\tilde h(\nabla_iR_{i\bar i},\eta_i^{\dagger_{\tilde h}}(\nabla_i\eta_i))\frac{\omega^n}{n!}\\
  =\ &3\int_Xh(\nabla_iR_{i\bar i},\Box\nabla_{ i}R_{i\bar i})\frac{\omega^n}{n!}-3\int_X\tilde h([R_{i\bar i},\nabla_i\eta_i],\nabla_i\eta_i)\frac{\omega^n}{n!}-6\int_X\tilde h(\mathbf{d}\nabla_iR_{i\bar i}, [\eta_i,R_{i\bar i}])\frac{\omega^n}{n!}.
\end{align*}
 Consequently, we arrive at
\begin{align*}
  A=\bigg(&-6\int_X\tilde h(\mathbf{d}\nabla_iR_{i\bar i},\mathbf{d}\nabla_{\bar i}R_{i\bar i})\frac{\omega^n}{n!} -12\mathrm{ Re}\int_X\tilde h(\mathbf{d}\nabla_iR_{i\bar i}, [\eta_i,R_{i\bar i}])\frac{\omega^n}{n!}\\
 &-5\int_X\tilde h([R_{i\bar i},\mathbf{d}^{\dagger_{\tilde h}}G[\eta_i\wedge \eta_i]],\mathbf{d}^{\dagger_{\tilde h}}G[\eta_i\wedge \eta_i])\frac{\omega^n}{n!}\bigg)\bigg|_{s_0}\\
  = \bigg(&-6\int_Xh(\nabla_i\mathbf{d}R_{i\bar i},\nabla_{i}\mathbf{d}R_{i\bar i})\frac{\omega^n}{n!} +6\int_X\tilde h([\eta_i,R_{i\bar i}], [\eta_i,R_{i\bar i}])\frac{\omega^n}{n!}\\
  &-5\int_X\tilde h([R_{i\bar i},\mathbf{d}^{\dagger_{\tilde h}}G[\eta_i\wedge \eta_i]],\mathbf{d}^{\dagger_{\tilde h}}G[\eta_i\wedge \eta_i])\frac{\omega^n}{n!}\bigg)\bigg|_{s_0}.
\end{align*}

Denote $G_i=\frac{\partial G}{\partial s^i},P_i=\frac{\partial P}{\partial s^i}, \Box_i=\frac{\partial \Box}{\partial s^i}$, then  the identity $G\Box+P=\mathrm{Id}$ leads to
\begin{align*}
  \mathbf{d}G_i\Box R_{i\bar i}&=-\mathbf{d}G\Box_iR_{i\bar i}-\mathbf{d}P_iR_{i\bar i}\\
  &=-\mathbf{d}\mathbf{d}^{\dagger_{\tilde h}}G[\eta_i, R_{i\bar i}]+[\eta_i, PR_{i\bar i}]\\
 &=-\mathbf{d}\mathbf{d}^{\dagger_{\tilde h}}G[\eta_i, R_{i\bar i}] ,
\end{align*}
where we apply the following identities for the second equality
\begin{align*}
  \Box_iA&=[\eta_i\wedge \mathbf{d}^{\dagger_{\tilde h}} A ]+\mathbf{d}^{\dagger_{\tilde h}}[\eta_i\wedge A],\\
  \mathbf{d}P_iA&=-[\eta_i\wedge PA].
\end{align*}
Therefore, we have
\begin{align*}
 ( \nabla_i\mathbf{d}R_{i\bar i})|_{s_0}&=(\nabla_i\mathbf{d}G\Box R_{i\bar i})|_{s_0}\\
  &=([\eta_i,G\Box R_{i\bar i}]+\mathbf{d}G_i\Box R_{i\bar i}+\mathbf{d}G\nabla_i\Box R_{i\bar i})|_{s_0}\\
  &=([\eta_i, R_{i\bar i}]-\mathbf{d}\mathbf{d}^{\dagger_{\tilde h}}G[\eta_i, R_{i\bar i}]-\mathbf{d}G\Lambda\cdot \nabla_i[\eta_i\wedge \eta_{\bar i}^{*_h}])|_{s_0}\\
  &=(P[\eta_i,R_{i\bar i}]-\mathbf{d}^{\dagger_{\tilde h}}G[\eta_i\wedge\mathbf{d}R_{i\bar i}]+\mathbf{d}G\Lambda\cdot[\mathbf{d}^{*_h}R_{i\bar i}\wedge\eta_i]+\mathbf{d}G\Lambda\cdot[\eta_{\bar i}^{*_h}\wedge\mathbf{d}^{\dagger_{\tilde h}}G[\eta_i\wedge \eta_i]])|_{s_0}\\
  &=(P[\eta_i,R_{i\bar i}]-\mathbf{d}^{\dagger_{\tilde h}}G[\eta_i\wedge\mathbf{d}R_{i\bar i}]-\mathbf{d}\mathbf{d}^{\dagger_{\tilde h}}G[\eta_i ,R_{i\bar i}]-\mathbf{d}\mathbf{d}^{\dagger_{\tilde h}}G\eta_{i}^{\dagger_{\tilde h}} G[\eta_i\wedge \eta_i])|_{s_0},
\end{align*}
which yields
\begin{align*}
&\ \bigg(\int_X\tilde h(\nabla_i\mathbf{d}R_{i\bar i},\nabla_{i}\mathbf{d}R_{i\bar i})\frac{\omega^n}{n!}\bigg)\bigg|_{s_0}\\
=&\ \bigg(\int_X\tilde h(P[\eta_i,R_{i\bar i}],P[\eta_i,R_{i\bar i}])\frac{\omega^n}{n!}+\int_X\tilde h(G[\eta_i\wedge\mathbf{d}R_{i\bar i}],[\eta_i\wedge\mathbf{d}R_{i\bar i}])\frac{\omega^n}{n!}\\
& \ \ \ +\int_X h(G\mathbf{d}^{\dagger_{\tilde h}}[\eta_i ,R_{i\bar i}],\mathbf{d}^{\dagger_{\tilde h}}[\eta_i ,R_{i\bar i}])\frac{\omega^n}{n!} +\int_X h(G\mathbf{d}^{\dagger_{\tilde h}}\eta_{i}^{\dagger_{\tilde h}} G[\eta_i\wedge \eta_i],\mathbf{d}^{\dagger_{\tilde h}}\eta_{i}^{\dagger_{\tilde h}} G[\eta_i\wedge \eta_i])\frac{\omega^n}{n!}\\
&\ \ \ +2\mathrm{Re}\int_X h(G\mathbf{d}^{\dagger_{\tilde h}}[\eta_i ,R_{i\bar i}],\mathbf{d}^{\dagger_{\tilde h}}\eta_{i}^{\dagger_{\tilde h}} G[\eta_i\wedge \eta_i])\frac{\omega^n}{n!}\bigg)\bigg|_{s_0}.
\end{align*}

In addition, we have
\begin{align*}
  &\int_X\tilde h([\eta_i,R_{i\bar i}],[\eta_i,R_{i\bar i}])\frac{\omega^n}{n!}-\int_X\tilde h(P[\eta_i,R_{i\bar i}],P[\eta_i,R_{i\bar i}])\frac{\omega^n}{n!}\\=&\int_X\tilde h(\Box G[\eta_i,R_{i\bar i}],\Box G[\eta_i,R_{i\bar i}])\frac{\omega^n}{n!}\\
  =&\int_X\tilde h(G[\eta_i\wedge\mathbf{d}R_{i\bar i}],[\eta_i\wedge\mathbf{d}R_{i\bar i}])\frac{\omega^n}{n!}
  +\int_Xh(G\mathbf{d}^{\dagger_{\tilde h}}[\eta_i,R_{i\bar i}],\mathbf{d}^{\dagger_{\tilde h}}[\eta_i,R_{i\bar i}])\frac{\omega^n}{n!}.
\end{align*}
Combining the above calculations, the desired equality follows.

(2) By means of Lemma \ref{4m}, one can easily check that
\begin{align*}
   \mathbf{ d}\nabla_i^2\eta_i+[\nabla_i\eta_i\wedge \eta_i]&=0,\\
   \mathbf{ d}^{\dagger_{\tilde h}}\nabla_i^2\eta_i&=0.
\end{align*}
Hence, we have
\begin{align*}
 \int_X\tilde h(\nabla_i^2\eta_i,\nabla_i^2\eta_i)\frac{\omega^n}{n!}& =\int_X\tilde h((P+G \mathbf{ d}^{\dagger_{\tilde h}}\mathbf{ d})\nabla_i^2\eta_i,\nabla_i^2\eta_i)\frac{\omega^n}{n!}\\
 &=\int_X\tilde h(P(\nabla_i^2\eta_i),P(\nabla_i^2\eta_i))\frac{\omega^n}{n!}+\int_X\tilde h(G \mathbf{ d}\nabla_i^2\eta_i,\mathbf{ d}\nabla_i^2\eta_i)\frac{\omega^n}{n!}\\
 &=\int_X\tilde h(P(\nabla_i^2\eta_i),P(\nabla_i^2\eta_i))\frac{\omega^n}{n!}+\int_X\tilde h(G ([\nabla_i\eta_i\wedge \eta_i]),[\nabla_i\eta_i\wedge \eta_i])\frac{\omega^n}{n!}.
\end{align*}

On the other hand, since $G^{\mathrm{WP}}$ is a K\"{a}hler metric, we always have $\frac{\partial^2 G^{\mathrm{WP}}_{k\bar l}}{\partial s^i\partial s^j}|_{s_0}=0$ for any $1\leq i,j,k,l\leq m$. Hence
\begin{align*}
   \Big(\frac{\partial^2}{\partial s^i\partial s^i}\int_X\tilde h(\eta_k,\eta_l)\frac{\omega^n}{n!}\Big)\Big|_{s_0} &=  \Big(\int_X\tilde h(\nabla_i^2\eta_k,\eta_l)\frac{\omega^n}{n!}+2\int_X\tilde h(\nabla_i\eta_k,\nabla_{\bar i}\eta_l)\frac{\omega^n}{n!}
  +\int_X\tilde h(\eta_k,\nabla^2_{\bar i}\eta_l)\frac{\omega^n}{n!}\Big)\Big|_{s_0}\\
  &=  \Big(\int_X\tilde h(\nabla_i^2\eta_k,\eta_l)\frac{\omega^n}{n!}+2\int_X\tilde h(\mathbf{d}^{\dagger_{\tilde h}}\nabla_i\eta_k,R_{l\bar i})\frac{\omega^n}{n!}
  +\int_X\tilde h(\mathbf{d}^{\dagger_{\tilde h}}\eta_k,\nabla_{\bar i}R_{l\bar i})\frac{\omega^n}{n!}\Big)\Big|_{s_0}\\
  &= \Big(\int_X\tilde h(\nabla_i^2\eta_k,\eta_l)\frac{\omega^n}{n!}\Big)\Big|_{s_0}\\
  &= 0,
\end{align*}
which means that $P((\nabla_i^2\eta_i)|_{s_0})=0$.
Therefore
\begin{align*}
 B=&\ \Big(\int_X\tilde h(G [\nabla_i\eta_i\wedge \eta_i],[\nabla_i\eta_i\wedge \eta_i])\frac{\omega^n}{n!}\Big)\Big|_{s_0}\\
 =& \ \bigg(\int_X\tilde h(G^2 [(\mathbf{d}\nabla_i\eta_i)\wedge \eta_i],[(\mathbf{d}\nabla_i\eta_i)\wedge \eta_i])\frac{\omega^n}{n!}+\int_X\tilde h(G^2 \mathbf{d}^{\dagger_{\tilde h}}[\nabla_i\eta_i\wedge \eta_i],\mathbf{d}^{\dagger_{\tilde h}}[\nabla_i\eta_i\wedge \eta_i])\frac{\omega^n}{n!}\bigg)\bigg|_{s_0}\\
 =&\ \bigg(\int_X\tilde h(G^2 [[\eta_i\wedge \eta_i]\wedge \eta_i],[[\eta_i\wedge \eta_i]\wedge \eta_i])\frac{\omega^n}{n!}
  +\frac{1}{4}\int_X\tilde h(G^2\nabla_i\Box\nabla_i\eta_i,\nabla_i\Box\nabla_i\eta_i)\frac{\omega^n}{n!}\bigg)\bigg|_{s_0}\\
 = &\ \bigg(\int_X\tilde h(G^2 [[\eta_i\wedge \eta_i]\wedge \eta_i],[[\eta_i\wedge \eta_i]\wedge \eta_i])\frac{\omega^n}{n!}+\frac{1}{4}\int_X\tilde h(G^2(\Box\nabla^2_i\eta_i+\Box_i\nabla_i\eta_i),\Box\nabla^2_i\eta_i+\Box_i\nabla_i\eta_i)\frac{\omega^n}{n!}\bigg)\bigg|_{s_0}\\
 = &\ \Big(\int_X\tilde h(G^2 [[\eta_i\wedge \eta_i]\wedge \eta_i],[[\eta_i\wedge \eta_i]\wedge \eta_i])\frac{\omega^n}{n!}\Big)\Big|_{s_0},
 \end{align*}
where the last equality follows from
\begin{align*}
 \Box\nabla^2_i\eta_i+\Box_i\nabla_i\eta_i=\mathbf{d}^{\dagger_{\tilde h}}\mathbf{d}\nabla^2_i\eta_i+\mathbf{d}^{\dagger_{\tilde h}}[\eta_i\wedge\nabla_i\eta_i]=0.
\end{align*}

(3) Again by the Ricci identity, we have
\begin{align*}
  &\int_X\tilde h(\nabla_{\bar i}\nabla_i\eta_i,\nabla_{\bar i}\nabla_i\eta_i)\frac{\omega^n}{n!}\\
 =&\int_X\tilde h(\nabla_i\mathbf{d}R_{i\bar i},\nabla_i\mathbf{d}R_{i\bar i})\frac{\omega^n}{n!}+\int_X\tilde h([\eta_i,R_{i\bar i}],[\eta_i,R_{i\bar i}])\frac{\omega^n}{n!}+2\mathrm{Re}\int_X\tilde h(\nabla_i\mathbf{d}R_{i\bar i},[\eta_i,R_{i\bar i}])\frac{\omega^n}{n!}
 \end{align*}
Then it follows from the calculations in (1) that
\begingroup
\allowdisplaybreaks
\begin{align*}
  C=&\ \bigg(\int_X\tilde h(P[\eta_i,R_{i\bar i}],P[\eta_i,R_{i\bar i}])\frac{\omega^n}{n!}+\int_X\tilde h(G[\eta_i\wedge\mathbf{d}R_{i\bar i}],[\eta_i\wedge\mathbf{d}R_{i\bar i}])\frac{\omega^n}{n!}\\
& \ \ \   +\int_X h(G\mathbf{d}^{\dagger_{\tilde h}}([\eta_i ,R_{i\bar i}]+\eta_{i}^{\dagger_{\tilde h}} G[\eta_i\wedge \eta_i]),\mathbf{d}^{\dagger_{\tilde h}}([\eta_i ,R_{i\bar i}]+\eta_{i}^{\dagger_{\tilde h}} G[\eta_i\wedge \eta_i]))\frac{\omega^n}{n!} \\
&\ \ \ + \int_X\tilde h(P[\eta_i,R_{i\bar i}],P[\eta_i,R_{i\bar i}])\frac{\omega^n}{n!}+\int_X\tilde h(G[\eta_i\wedge\mathbf{d}R_{i\bar i}],[\eta_i\wedge\mathbf{d}R_{i\bar i}])\frac{\omega^n}{n!}\\
&\ \ \ +\int_X h(G\mathbf{d}^{\dagger_{\tilde h}}[\eta_i ,R_{i\bar i}],\mathbf{d}^{\dagger_{\tilde h}}[\eta_i ,R_{i\bar i}])\frac{\omega^n}{n!}\\
&\ \ \ + 2\int_X\tilde h(P[\eta_i,R_{i\bar i}],P[\eta_i,R_{i\bar i}])\frac{\omega^n}{n!}+2\int_X\tilde h(G[\eta_i\wedge\mathbf{d}R_{i\bar i}],[\eta_i\wedge\mathbf{d}R_{i\bar i}])\frac{\omega^n}{n!}\\
&\ \ \ -2\int_X h(G\mathbf{d}^{\dagger_{\tilde h}}[\eta_i ,R_{i\bar i}],\mathbf{d}^{\dagger_{\tilde h}}[\eta_i ,R_{i\bar i}])\frac{\omega^n}{n!}\\
&\ \ \ -2\mathrm{Re}\int_X h(G\mathbf{d}^{\dagger_{\tilde h}}[\eta_i ,R_{i\bar i}],\mathbf{d}^{\dagger_{\tilde h}}\eta_{i}^{\dagger_{\tilde h}} G[\eta_i\wedge \eta_i])\frac{\omega^n}{n!}\bigg)\bigg|_{s_0}\\
 =&\ \bigg(4\int_X\tilde h(P[\eta_i,R_{i\bar i}],P[\eta_i,R_{i\bar i}])\frac{\omega^n}{n!}+4\int_X\tilde h(G[\eta_i\wedge\mathbf{d}R_{i\bar i}],[\eta_i\wedge\mathbf{d}R_{i\bar i}])\frac{\omega^n}{n!} \ \ \\
&\ \ \  +\int_X h(G\mathbf{d}^{\dagger_{\tilde h}}\eta_{i}^{\dagger_{\tilde h}} G[\eta_i\wedge \eta_i],\mathbf{d}^{\dagger_{\tilde h}}\eta_{i}^{\dagger_{\tilde h}} G[\eta_i\wedge \eta_i])\frac{\omega^n}{n!}\bigg)\bigg|_{s_0}.
\end{align*}
\endgroup
The desired equality follows.

(4) We have
\begin{align*}
 \Big(\int_X\tilde h(\nabla_i(P\nabla_i \eta_i),\nabla_i(P\nabla_i \eta_i))\frac{\omega^n}{n!}\Big)\Big|_{s_0}&= \Big(\int_X\tilde h(P_iG\mathbf{d}^{\dagger_{\tilde h}}([\eta_i\wedge \eta_i]),P_iG\mathbf{d}^{\dagger_{\tilde h}}([\eta_i\wedge \eta_i]) )\frac{\omega^n}{n!}\Big)\Big|_{s_0}\\
 &= \Big(\int_X\tilde h(P\mathbf{d}^{\dagger_{\tilde h}}\nabla_i(G[\eta_i\wedge \eta_i]),P\mathbf{d}^{\dagger_{\tilde h}}\nabla_i(G[\eta_i\wedge \eta_i]))\frac{\omega^n}{n!}\Big)\Big|_{s_0}\\
 &= 0,
\end{align*}
where the second equality is due to the identity
$$P_i\mathbf{d}^{\dagger_{\tilde h}}+P\nabla_i\mathbf{d}^{\dagger_{\tilde h}}=0.$$
And similarly,  we have
\begin{align*}
  & \Big(\int_X\tilde h(\nabla_{\bar i}(P\nabla_i \eta_i),\nabla_{\bar i}(P\nabla_i \eta_i))\frac{\omega^n}{n!}\Big)\Big|_{s_0}\\
   =\ & \Big(\int_X\tilde h(P\nabla_{\bar i}\nabla_i\eta_i,P\nabla_{\bar i}\nabla_i\eta_i)\frac{\omega^n}{n!}+\int_X\tilde h(P\nabla_{\bar i}\mathbf{d}^{\dagger_{\tilde h}}G[\eta_i\wedge \eta_i],P\nabla_{\bar i}\mathbf{d}^{\dagger_{\tilde h}}G[\eta_i\wedge \eta_i] )\frac{\omega^n}{n!}\Big)\Big|_{s_0}\\
   =\ &\Big(4\int_X\tilde h(P[\eta_i,R_{i\bar i}],P[\eta_i,R_{i\bar i}])\frac{\omega^n}{n!} +\int_X\tilde h(P\eta_i^{\dagger_{\tilde h}}G[\eta_i\wedge \eta_i],P\eta_i^{\dagger_{\tilde h}}G[\eta_i\wedge \eta_i] )\frac{\omega^n}{n!}\Big)\Big|_{s_0}\\
 =\ &\bigg(4\int_X\tilde h(P[\eta_i,G\eta_i^{\dagger_{\tilde h}}\eta_i
],P[\eta_i,G\eta_i^{\dagger_{\tilde h}}\eta_i
])\frac{\omega^n}{n!} \\
&\ \ \   +\int_X\tilde h(P\mathbf{d}^{\dagger_{\tilde h}}[\eta_i\wedge G^2\mathbf{d}^{\dagger_{\tilde h}}[\eta_i\wedge\eta_i]],P\mathbf{d}^{\dagger_{\tilde h}}[\eta_i\wedge G^2\mathbf{d}^{\dagger_{\tilde h}}[\eta_i\wedge\eta_i]])\frac{\omega^n}{n!}\bigg)\bigg|_{s_0}\\
 =\ &\Big(4\int_X\tilde h(P[\eta_i,G\eta_i^{\dagger_{\tilde h}}\eta_i
],P[\eta_i,G\eta_i^{\dagger_{\tilde h}}\eta_i
])\frac{\omega^n}{n!}\Big)\Big|_{s_0}.
\end{align*}
Then the desired equality is obtained.

(5) Since we have \begin{align*}
\Big(\int_X\tilde h(\nabla_i^2\eta_i,\nabla_i\eta_i)\frac{\omega^n}{n!}\Big)\Big|_{s_0}&=\Big(\int_X\tilde h(G \mathbf{ d}^{\dagger_{\tilde h}}\mathbf{ d}\nabla_i^2\eta_i,\nabla_i\eta_i)\frac{\omega^n}{n!}\Big)\Big|_{s_0}\\
 &=\Big(\int_X\tilde h(G ([\nabla_i\eta_i\wedge \eta_i]),[\eta_i\wedge \eta_i])\frac{\omega^n}{n!}\Big)\Big|_{s_0}\\
 &=-\Big(\int_X\tilde h([(\mathbf{d}^{\dagger_{\tilde h}}G[\eta_i\wedge \eta_i])\wedge \eta_i],G[\eta_i\wedge \eta_i])\frac{\omega^n}{n!}\Big)\Big|_{s_0},\\
\Big(\int_X\tilde h(\nabla_{\bar i}\nabla_i\eta_i,\nabla_i\eta_i)\frac{\omega^n}{n!}\Big)\Big|_{s_0}&=\Big(\int_X\tilde h([\eta_i,R_{i\bar i}]+\nabla_i\mathbf{d}R_{i\bar i}, \nabla_i\eta_i)\frac{\omega^n}{n!}\Big)\Big|_{s_0}\\
  &=-2\Big(\int_X\tilde h([\eta_i\wedge  \mathbf{d}G\eta_i^{\dagger_{\tilde h}}\eta_i],G[\eta_i\wedge \eta_i])\frac{\omega^n}{n!}\Big)\Big|_{s_0},
\end{align*}
the desired equality follows.
\end{proof}

\begin{theorem}\label{finslercur}
Let $\{s^1,\cdots,s^m\}$ be the local holomorphic coordinate on $S$.
The holomorphic sectional curvature of the Finsler metric $F_\kappa$
 is given by
\begin{align*}
 K_{F_\kappa}\Big(\frac{\partial}{\partial s^i}\Big)
 = &\ \bigg(\Big(\int_X\tilde h(\eta_i,\eta_i)\frac{\omega^n}{n!}\Big)^2+\kappa\int_X\tilde h(\mathbf{d}^{\dagger_{\tilde h}}G[\eta_i\wedge \eta_i], \mathbf{d}^{\dagger_{\tilde h}}G[\eta_i\wedge \eta_i])\frac{\omega^n}{n!}\bigg)^{-\frac{3}{2}}\\
 &\cdot\bigg[2\Big(\int_X\tilde h(\eta_i,\eta_i)\frac{\omega^n}{n!}\Big)\Big(2\int_X h(\eta_i^{\dagger_{\tilde h}}\eta_i,G\eta_i^{\dagger_{\tilde h}}\eta_i)\frac{\omega^n}{n!}-\int_X\tilde h(\mathbf{d}^{\dagger_{\tilde h}}G[\eta_i\wedge \eta_i], \mathbf{d}^{\dagger_{\tilde h}}G[\eta_i\wedge \eta_i])\frac{\omega^n}{n!}\Big)\\
 &\ \ \ \ +\kappa\Big(-5\int_X\tilde h([G\eta_i^{\dagger_{\tilde h}}\eta_i,\mathbf{d}^{\dagger_{\tilde h}}G[\eta_i\wedge \eta_i]],\mathbf{d}^{\dagger_{\tilde h}}G[\eta_i\wedge \eta_i])\frac{\omega^n}{n!}\\
  &\ \ \ \ \ \ \ \ \ \ \ +5\int_Xh(G\mathbf{d}^{\dagger_{\tilde h}}\eta_{ i}^{\dagger_{\tilde h}}G[\eta_i\wedge \eta_i],\mathbf{d}^{\dagger_{\tilde h}}\eta_{ i}^{\dagger_{\tilde h}}G[\eta_i\wedge \eta_i])\frac{\omega^n}{n!}\\
 &\ \ \ \ \ \ \ \ \ \ \ -12\mathrm{Re}\int_X h(G\mathbf{d}^{\dagger_{\tilde h}}[\eta_i,G\eta_i^{\dagger_{\tilde h}}\eta_i],\mathbf{d}^{\dagger_{\tilde h}}\eta_{i}^{\dagger_{\tilde h}} G[\eta_i\wedge \eta_i])\frac{\omega^n}{n!}\\ &\ \ \ \ \ \ \ \ \ \ \ -4\int_X\tilde h(G[\eta_i\wedge\mathbf{d}G\eta_i^{\dagger_{\tilde h}}\eta_i],[\eta_i\wedge\mathbf{d}G\eta_i^{\dagger_{\tilde h}}\eta_i])\frac{\omega^n}{n!}\\
  &\ \ \ \ \ \ \ \ \ \ \ -\int_X\tilde h(G^2 [[\eta_i\wedge \eta_i]\wedge \eta_i],[[\eta_i\wedge \eta_i]\wedge \eta_i])\frac{\omega^n}{n!}\Big)\bigg]\\
  &+\kappa^2\bigg(\Big(\int_X\tilde h(\eta_i,\eta_i)\frac{\omega^n}{n!}\Big)^2+\kappa\int_X\tilde h(\mathbf{d}^{\dagger_{\tilde h}}G[\eta_i\wedge \eta_i], \mathbf{d}^{\dagger_{\tilde h}}G[\eta_i\wedge \eta_i])\frac{\omega^n}{n!}\bigg)^{-\frac{5}{2}}\\
 &\ \ \ \cdot \bigg|\int_X\tilde h( [(\mathbf{d}^{\dagger_{\tilde h}}G[\eta_i\wedge \eta_i])\wedge\eta_i],G[\eta_i\wedge \eta_i])\frac{\omega^n}{n!}+2\int_X\tilde h([\eta_i\wedge \mathbf{d} G\eta_i^{\dagger_{\tilde h}}\eta_i],G[\eta_i\wedge \eta_i])\frac{\omega^n}{n!}\bigg|^2.
\end{align*}
\end{theorem}

\begin{proof}
Substitute the   formulas in Lemma \ref{lo} into the following expression
\begin{align*}
  &-\bigg(\frac{\partial^2}{\partial s^i\partial \bar s^i}\log\Big((\int_X\tilde h(\eta_i,\eta_i)\frac{\omega^n}{n!})^2+\kappa\int_X\tilde h([\eta_i\wedge\eta_i],G[\eta_i\wedge\eta_i]\frac{\omega^n}{n!}\Big)\bigg)\bigg|_{s_0}\\
  =&\frac{(2(\int_X\tilde h(\eta_i,\eta_i)\frac{\omega^n}{n!})(2\int_X h(\eta_i^{\dagger_{\tilde h}}\eta_i,G\eta_i^{\dagger_{\tilde h}}\eta_i)\frac{\omega^n}{n!}-\int_X\tilde h([\eta_i\wedge\eta_i],G[\eta_i\wedge\eta_i]))\frac{\omega^n}{n!}))|_{s_0}-\kappa(A+B+C-D)}{((\int_X\tilde h(\eta_i,\eta_i)\frac{\omega^n}{n!})^2+\kappa\int_X\tilde h([\eta_i\wedge\eta_i],G[\eta_i\wedge\eta_i])\frac{\omega^n}{n!})|_{s_0}}\\
  &+\frac{\kappa^2E}{((\int_X\tilde h(\eta_i,\eta_i)\frac{\omega^n}{n!})^2+\kappa\int_X\tilde h([\eta_i\wedge\eta_i],G[\eta_i\wedge\eta_i])\frac{\omega^n}{n!})^2|_{s_0}},
\end{align*}
and then by the  covariance we get the theorem.
\end{proof}

\bigskip

\noindent\textbf{Acknowledgements}. 
The author P. Huang acknowledges funding by the Deutsche Forschungsgemeinschaft (DFG, German Research Foundation) under Germany's Excellence Strategy EXC 2181/1 - 390900948 (the Heidelberg  STRUCTURES Excellence Cluster), and by Collaborative Research Center/Transregio (CRC/TRR 191; 281071066-TRR 191).
The authors would like to express their deep gratitude to  the anonymous referee for many valuable suggestions.

\end{document}